\documentclass[11pt, a4paper,reqno]{amsart}
\pdfoutput=1
\usepackage{tikz,enumitem,rotating,latexsym,bm,stmaryrd,caption,}
\usetikzlibrary{positioning,intersections,decorations,shadings}
\usetikzlibrary{shapes}
\usetikzlibrary{arrows}
\usetikzlibrary{patterns}
\usepackage{tkz-euclide}
\usetikzlibrary{calc}
\usetikzlibrary{shapes.geometric}
\tikzset{snake it/.style={decorate, decoration=snake}}
\tikzset{zigzag/.style={decorate, decoration=zigzag}}

\definecolor{ao(english)}{rgb}{0.0, 0.5, 0.0}
\definecolor{darkgreen}{rgb}{0.0, 0.5, 0.0}

\pgfdeclareverticalshading{rainbow}{100bp}
{color(0bp)=(orange);color(25bp)=(orange);
color(50bp)=(magenta); color(75bp)=(cyan);
 color(100bp)=(cyan)}

\usetikzlibrary{decorations.pathmorphing}

	\definecolor{eng}{rgb}{0.0, 0.5, 0.0}
\definecolor{apple}{rgb}{0.55, 0.71, 0.0}
\definecolor{cadmium}{rgb}{0.0, 0.42, 0.24}
\definecolor{darkspringgreen}{rgb}{0.09, 0.45, 0.27}
\definecolor{amethyst}{rgb}{0.6, 0.4, 0.8}
\definecolor{ao}{rgb}{0.0, 0.0, 1.0}
\definecolor{atomictangerine}{rgb}{1.0, 0.6, 0.4}
\definecolor{carmine}{rgb}{0.59, 0.0, 0.09}

\definecolor{toggle}{rgb}{1.0, 0.94, 0.96}

\usepackage[normalem]{ulem}

\DeclareMathOperator{\rad}{rad}
\DeclareMathOperator{\soc}{soc}

\def\down{\vee}
\def\up{\wedge}

\usepackage{standalone}

\usepackage{etex}
 
\tikzset{
  variable line width/.style={
    every variable line width/.append style={#1},
    to path={%
      \pgfextra{%
        \draw[every variable line width/.try,line width=\pgfkeysvalueof{/tikz/thickness}] (\tikztostart) -- (\tikztotarget);
      }%
      (\tikztotarget)
    },
  },
  thickness/.initial=0.6pt,
  every variable line width/.style={line cap=round, line join=round},
}

\usepackage{todonotes}

\usepackage{tikz}
\usetikzlibrary{matrix,  intersections, calc, decorations.pathreplacing} 

\usepackage{tikz-cd}

\newlength{\superthick}
\newlength{\cornerradius}
\tikzstyle{corner}=[rounded corners=\cornerradius]
\tikzstyle{dot}=[circle, inner sep=0pt, minimum size=4.8pt]
\tikzstyle{string}=[line width=\superthick]
\tikzstyle{std}=[string,dash pattern=on 0.9pt off 0.9pt]
\definecolor{realcyan}{rgb}{0,1,1}

 \usepackage{amsmath,amsthm,amsfonts,amssymb,mathrsfs,pb-diagram}
\usepackage[
bookmarks=true,colorlinks=true,linktoc=page,citecolor=green!80!black,linkcolor=green!80!black,urlcolor=green!80!black]{hyperref}
\usepackage{caption}
\usepackage{lipsum,wasysym}
\usepackage{mathtools}
\usepackage[a4paper,margin=2.6cm]{geometry}
\usepackage{cleveref}
 \usepackage{amsmath}
\mathchardef\mhyphen="2D
\usepackage{color}
\usepackage{xcolor}
\usepackage{ifthen}
\usepackage{sidecap}   
\definecolor{mediumblue}{rgb}{0.0, 0.0, 0.8}

\renewcommand{\geq}{\geqslant}
\renewcommand{\leq}{\leqslant}

\tikzset{wei/.style= 
{red,double=red,double
distance=0.5pt}}

\makeatletter
\newcommand{\ostar}{\mathbin{\mathpalette\make@circled\star}}
\newcommand{\make@circled}[2]{%
  \ooalign{$\m@th#1\smallbigcirc{#1}$\cr\hidewidth$\m@th#1#2$\hidewidth\cr}%
}
\newcommand{\smallbigcirc}[1]{%
  \vcenter{\hbox{\scalebox{0.77778}{$\m@th#1\bigcirc$}}}%
}
\makeatother

\tikzset{wei2/.style={red,double=red,double
distance=0.5pt}}

\allowdisplaybreaks
\numberwithin{equation}{section}
\usepackage{scalefnt}

\newtheorem{thm}{Theorem}[section]
\newtheorem{cor}[thm]{Corollary}

\newtheorem{lem}[thm]{Lemma}

\newtheorem{prop}[thm]{Proposition}

\newtheorem*{ack}{Acknowledgements}

\newtheorem*{prop*}{Proposition}
\newtheorem*{thmA}{Theorem A}

\newtheorem*{cor*}{Corollary}

\newtheorem*{conj*}{Conjecture D}

\newtheorem*{conj1*}{Conjecture B}
\newtheorem*{Acknowledgements*}{Acknowledgements}

\theoremstyle{rmk}

\theoremstyle{defn}
\newtheorem{rmk}[thm]{Remark}
\newtheorem{defn}[thm]{Definition}
\newtheorem{eg}[thm]{Example}

\newcommand{\la}{\lambda}

\newcommand{\ZZ}{{\mathbb Z}}
\newcommand{\NN}{{\mathbb N}}

\newcommand{\Amod}{A\text{--}{\rm mod}}
\newcommand{\Bmod}{B\text{--}{\rm mod}}
\newcommand{\Aproj}{A\text{--}{\rm proj}}
\newcommand{\Atilt}{A\text{--}{\rm tilt}}
\newcommand{\Kmod}{K^m_n\text{--}{\rm mod}}
\newcommand{\Kproj}{K^m_n\text{--}{\rm proj}}
\newcommand{\Hmod}{H^m_n\text{--}{\rm mod}}

\DeclareMathOperator{\Hom}{Hom}

\let\<=\langle
\let\>=\rangle

\tikzset{
ultra thin/.style= {line width=0.05pt},
very thin/.style=  {line width=0.2pt},
thin/.style=       {line width=0.1pt},
semithick/.style=  {line width=0.6pt},
thick/.style=      {line width=0.8pt},
very thick/.style= {line width=1.2pt},
ultra thick/.style={line width=1.6pt}
}

\crefname{ques}{Question}{Questions}
\crefname{defn}{Definition}{Definitions}
\crefname{thm}{Theorem}{Theorems}
\crefname{prop}{Proposition}{Propositions}
\crefname{lem}{Lemma}{Lemmas}
\crefname{cor}{Corollary}{Corollaries}
\crefname{conj}{Conjecture}{Conjectures}
\crefname{section}{Section}{Sections}
\crefname{subsection}{Subsection}{Subsections}
\crefname{eg}{Example}{Examples}
\crefname{figure}{Figure}{Figures}
\crefname{rem}{Remark}{Remarks}
\crefname{rmk}{Remark}{Remarks}
\crefname{equation}{equation}{equation}

\Crefname{ques}{Question}{Questions}
\Crefname{defn}{Definition}{Definitions}
\Crefname{thm}{Theorem}{Theorems}
\Crefname{prop}{Proposition}{Propositions}
\Crefname{lem}{Lemma}{Lemmas}
\Crefname{cor}{Corollary}{Corollaries}
\Crefname{conj}{Conjecture}{Conjectures}
\Crefname{section}{Section}{Sections}
\Crefname{subsection}{Subsection}{Subsections}
\Crefname{eg}{Example}{Examples}
\Crefname{figure}{Figure}{Figures}
\Crefname{rem}{Remark}{Remarks}
\Crefname{rmk}{Remark}{Remarks}

\usepackage[hang,flushmargin]{footmisc}

\begin{document}

 \title[Faithful covers of Khovanov arc algebras]{
Faithful covers of Khovanov arc algebras
  }

 \author{C. Bowman}
       \address{Department of Mathematics, 
University of York, Heslington, York,  UK}
\email{chris.bowman-scargill@york.ac.uk}
  
 \author{M. De Visscher}
	\address{Department of Mathematics, City, University of London,   London, UK}
\email{maud.devisscher.1@city.ac.uk}

 \author{A. Dell'Arciprete}
       \address{Department of Mathematics, 
University of York, Heslington, York,  UK}
\email{alice.dellarciprete@york.ac.uk}

		\author{A.  Hazi}
	 
     \address{School of Mathematics, University of Leeds, Leeds, LS2 9JT}
\email{a.hazi@leeds.ac.uk}

		\author{R. Muth}
	 
     \address{Department of Mathematics and Computer Science,
Duquesne University,
 Pittsburgh, PA 15282}
\email{muthr@duq.edu}

		\author{ C. Stroppel}
 \address{ Mathematical Institute, Endenicher Allee 60, 53115 Bonn}
 \email{stroppel@math.uni-bonn.de}
 
 \maketitle

\begin{abstract}
We show that the extended Khovanov algebra $K^m_n$ is an $(|n-m|-1)$-faithful cover of the Khovanov arc algebra $H^m_n$.
  \end{abstract}

\section{Introduction}

 The  Khovanov arc algebras,  $H^m_n$, were first  introduced  by Khovanov (in the case $m=n$) in his  pioneering   construction of homological knot invariants for  tangles \cite{MR1740682,MR1928174}.   These homological  knot invariants have  subsequently been developed by Rasmussen and put to use in   Piccirillo's proof that the Conway knot is not slice \cite{MR2729272,MR4076631}. 
  The Khovanov arc algebras and their quasi-hereditary covers,  $K^m_n$,  
  have   been studied  from the point of view of  symplectic geometry \cite{MR4422212}
   and representation theory  \cite{MR2600694,MR2918294,MR2781018,MR2955190,MR2881300,BarWang},
  and  they  provide the  
exciting possibility of constructing  algebraic invariants suitable for Crane--Frenkel's  approach to the smooth 4-dimensional Poincar\'e  conjecture \cite{Manolescu}.

For $m,n \in \NN$, the Khovanov arc algebra  $H^m_n $    is a  symmetric algebra  and hence has infinite global dimension. 
These algebras are best understood by way of their     covers, the extended arc algebras   $K^m_n $ for $m,n \in \NN$. 
The extended arc algebras are Koszul, quasi-hereditary algebras, and therefore they possess standard modules, 
have finite global dimension, and  rigid cohomological structure (for example their  
 radical structures can be encoded combinatorially via the grading).  
 We wish to understand the limits of what cohomological information can be passed back-and-forth between the  Khovanov arc algebras and their 
quasi-hereditary  covers by way of the Schur functor $f:  K^m_n{\rm -mod} \to  H^m_n{\rm -mod}$ and its inverse.  
 Rouquier's language of ``faithfulness" of quasi-hereditary covers  allows us to  address this question for important  subcategories of 
$ K^m_n{\rm -mod} $ and $ H^m_n{\rm -mod}$, namely the subcategories of modules possessing   standard/cell filtrations (see \cref{details} for more details).
Throughout this paper we work over an arbitrary field $\Bbbk$ of characteristic $p\geq 0$.

\begin{thmA} 
The extended arc algebras   $K^m_n $ are $(|n-m|-1)$-faithful covers of the 
Khovanov arc algebras $H^m_n $ for $m,n \in \NN$.  In other words,
$$
{\rm Ext}^i_{K^m_n}	(M,N) \cong
{\rm Ext}^i_{H^m_n}	(f(M),f(N))  
$$
for $M,N$ a pair of standard-filtered modules and $0\leq i \leq  |m-n|-1$. 
\end{thmA}

It is worth noting that  
the cohomological connection of Theorem A  becomes stronger as $|n-m|$ increases.
 We know of only one other similar instance of this phenomenon:  over a field of characteristic $p$,
 the classical  Schur algebra is a $(p-3)$-faithful cover of the group algebra of the symmetric group,
\cite[Corollary 3.9.2]{MR2050037} and \cite{donkinhandbook}; 
that is,  the cohomological connection becomes stronger 
as the characteristic of the field increases. 

To the authors' knowledge, all other results concerning  faithfulness of quasi-hereditary covers
 concern only 0 and $1$-faithfulness;  the most famous of which are 
 Rouquier--Shan--Varagnolo--Vasserot's proof that, under mild conditions on the parameters, the cyclotomic Hecke algebras have 0-faithful 
 covers over $\mathbb C$  (the categories $\mathcal{O}$ for  cyclotomic rational double affine Hecke algebras \cite{RSVV}),
  and Webster's extension of this result to arbitrary ground fields by way of the weighted KLR algebras \cite{MR3732238}.

\medskip
\noindent \textbf{Structure of the paper.} Sections 2, 3 and 4 contain all the necessary background for this paper. Section 2 recalls the notions of highest weight category, Schur functors and quasi-hereditary covers. Section 3 contains the combinatorics of weights, cups and caps needed to define the (extended) arc algebras and study their representation theory. Section 4 introduces the extended arc algebra $K^m_n$ and its quasi-hereditary structure.  Here we also introduce projective functors relating the module categories of extended arc algebras of different ranks which play a crucial role in most of our proofs. Section 5 is new and gives an inductive construction of the indecomposable tilting modules for the extended arc algebras (see \cref{tilting}). In Section 6, we introduce the Khovanov arc algebra $H^m_n$ as an idempotent truncation of $K^m_n$. We also study analogues of the projective functors for $H^m_n$. In Section 7, we prove that $K^m_n$ is a $0$-faithful cover of $H^m_n$ when $m\neq n$ (see \cref{0faithful}). Finally, Section 8 contains the proof of Theorem A (see \cref{ifaithful}).

\begin{ack}
We would like to thank Steve Donkin and Harry Geranios for interesting and informative conversations during the writing of this paper. 
The  authors are grateful for 
  financial support from EPSRC grant EP/V00090X/1 (first and third authors) and EPSRC Programme Grant EP/W007509/1 (fourth author).
\end{ack}

\section{Background: Highest weight categories and Schur functors}\label{details}

In this section we recall the abstract framework required for this paper, that  of quasi-hereditary covers. 
All of the material in this section can be found in the excellent references \cite[Appendix]{Donkin} and \cite[Section 4]{ROUQ}.

\subsection{Highest weight categories}
We start by recalling the notion of highest weight categories introduced by Cline, Parshall and Scott in \cite{MR961165}.

Let $\Bbbk$ be a field and $A$ be a finite dimensional $\Bbbk$-algebra. We will be working with the category $\Amod$ of finite dimensional left $A$-modules.  Throughout the paper, we will identify isomorphic modules.
We define  the  {\sf   radical}  of a finite-dimensional $A$-module $M$, denoted
$\rad M$, to be the smallest submodule of $M$ such that the corresponding
quotient is semisimple. 
The radical series is given by setting $\rad^0 M = M$ and $\rad^i M = \rad (\rad^{i-1} M)$ for $i\geq 1$ and the radical layers $\rad_i M$ are the semisimple subquotients $\rad_i M = \rad^i(M) / \rad^{i+1}(M)$ for $i\geq 0$. We define the {\sf   socle}
of a finite-dimensional $A$-module $M$, denoted
$\soc M$, to be the largest semisimple submodule of $M$. 
The socle series is defined by setting $\soc^0 M = \{0\}$ and $\soc^{i+1} M= \pi_i^{-1} (\soc(M/\soc^i M))$ for $i\geq 0$ where $\pi_i : M \rightarrow M/\soc^i M$ is the natural projection. We define the socle layers $\soc_i M$ to be the semisimple subquotients  $\soc_i M = \soc^i(M)/\soc^{i-1}(M)$ for $i\geq 1$.

Let $(\Lambda , \leq)$ be a poset indexing the isomorphism classes of simple $A$-modules. For each $\la \in \Lambda$, we denote the corresponding simple $A$-module by $L(\la)$ and its projective cover by $P(\la)$. So we have $\rad_0 P(\la)=P(\la) /\rad^1 P(\la) = L(\la)$. We write $\Aproj$ for the subcategory of $\Amod$ whose objects are the projective $A$-modules.

For any $\pi \subseteq \Lambda$ and $M\in \Amod$ we say that $M$ belongs to $\pi$ if all its composition factors belong to $\{L(\mu) \, : \, \mu\in \pi\}$. For any $M\in \Amod$ we define $O^\pi(M)$ to be the unique minimal submodule of $M$ such that $M/O^\pi(M)$ belongs to $\pi$.  Now, for $\la \in \Lambda$, set $\pi(\la) =\{ \mu \in \Lambda, \, \mu<\la\}$ and define the {\sf standard module} $\Delta(\la)$ by 
$$\Delta (\la) = P(\la) / O^{\pi(\la)}(\rad P(\la)).$$
For $\la,\mu \in  \Lambda$, 
we write  $[\Delta(\la) : L(\mu)] $ for the composition factor multiplicity of $L(\mu)$ in $\Delta(\la)$. 
By construction we have that  $[\Delta(\la) : L(\mu)] \neq 0$ implies $\mu \leq \la$ and $[\Delta(\la), L(\la)] = 1$. We let 
$$\Delta = \{ \Delta(\la) \mid  \la \in \Lambda\}.$$
For $M\in \Amod$, we say that $M$ has a $\Delta$-filtration if we have a filtration of $A$-modules
$$0 = M_0 \subset M_1 \subset \ldots \subset M_n = M$$
such that $M_i/M_{i-1}$ is isomorphic to a standard module  for all $1\leq i \leq n$. We let 
  $( \Amod )^\Delta$ denote the subcategory of $\Amod$ whose objects admit a $\Delta$-filtration. 
For $M\in (\Amod)^\Delta$ and $\la \in \Lambda$ we write $(M: \Delta(\la))$ for the multiplicity of $\Delta(\la)$ as a section in a $\Delta$-filtration of $M$. Note that this does not depend on the choice of filtration. 
We can now give the main definition for this section.

\begin{defn}
We say that $A$ is  a {\sf quasi-hereditary algebra}  and that the  category $\Amod$ is   a {\sf highest weight category} (with respect to the poset $(\Lambda , \leq)$) if for all $\la \in \Lambda$ we have
\begin{itemize}[leftmargin=*]
\item $P(\la) \in (\Amod)^\Delta$,
\item $(P(\la):\Delta(\la))  = 1$, and 
\item $(P(\la):\Delta(\mu)) \neq 0$ implies $\mu \geq \la$.
\end{itemize}
\end{defn}

We will assume from now on that $\Amod$ is a highest weight category.  The following proposition is well known.

\begin{prop}\label{Deltakernel}
Let $0\rightarrow M_1 \rightarrow M_2 \rightarrow M_3 \rightarrow 0$ be an exact sequence in $\Amod$. If $M_2$ and $M_3$ belong to $(\Amod)^\Delta$ then so does $M_1$.
\end{prop}

We will assume further that the algebra $A$ is endowed with an anti-automorphism $^*:A\rightarrow A$. This gives a duality functor
$$^{\ostar} : \Amod \rightarrow  \Amod \, :\, M\mapsto M^{\ostar} = \Hom_A(M, \Bbbk) $$where for $a\in A$ and $\psi\in M^{\ostar}$ we have $(a\psi)(m) = \psi(a^*m)$ for all $m\in M$. We will assume further that $L(\la)^{\ostar} \cong L(\la)$ for all $\la \in \Lambda$.
Then we have the following reciprocity.

\begin{thm}[Brauer--Humphreys reciprocity] \label{BH}
Let $A$ be a quasi-hereditary algebra endowed with an anti-automorphism  $^*:A\rightarrow A$. 
For any $\la , \mu \in \Lambda$ we have 
$$(P(\la):\Delta(\mu)) = [\Delta(\mu): L(\la)].$$
\end{thm}

We call $T\in \Amod$ a {\sf tilting module} if $T$ and $T^{\ostar}$ belong to $(\Amod)^\Delta$ and write $\Atilt$ for the subcategory of $\Amod$ whose objects are tilting modules.

\begin{thm}[\cite{MR1128706}]
 For each $\la\in \Lambda$ there is an indecomposable tilting module $T(\la)$ with the following properties
\begin{itemize}
\item $(T(\la):\Delta(\la))  = 1$, and 
\item $(T(\la): \Delta(\mu)) \neq 0$ implies $\mu \leq \la$.
\end{itemize}
The set  $\{T(\la) \, : \, \la \in \Lambda\}$ is a complete set of pairwise non-isomorphic indecomposable tilting modules.
\end{thm}

\subsection{Schur functors and covers}

Now let $e=e^2\in A$ be an idempotent in $A$ and set $B = eAe$ to be the corresponding idempotent subalgebra. We have the Schur functor
\begin{align*}f	&: \Amod \rightarrow \Bmod \, : \, M \mapsto eM 
\intertext{and the  inverse Schur functors}
g &: \Bmod \rightarrow \Amod \, : \, N \mapsto \Hom_B(eA, N),
\intertext{and }
 \tilde{g} &: \Bmod \rightarrow \Amod \, : \, N \mapsto Ae\otimes_B N.\end{align*}
The functor $f$ is exact, the functor $g$ is left exact and the functor $\tilde{g}$ is right exact. Moreover, $f$ is left adjoint to $g$ and right adjoint to $\tilde{g}$ i.e.
$$\Hom_B(fM, N) \cong \Hom_A(M,gN) \quad \mbox{and} \quad \Hom_B(N, fM) \cong \Hom_A(\tilde{g}N,M)$$
for all $M\in \Amod$, $N\in \Bmod$. We have the corresponding unit $\varepsilon : fg\rightarrow {\rm Id}$ and counit $\eta: {\rm Id} \rightarrow gf$. The unit $\varepsilon$ is an isomorphism.

\begin{prop}[{\cite[Proposition 4.33]{ROUQ}}] The following statements are equivalent.
\begin{itemize}[leftmargin=*]
\item $A\cong {\rm End}_B (eA)$.
\item The map $\eta(M) : M \rightarrow gf(M)$ is an isomorphism for all $M\in \Aproj$.
\item The functor $f$ restricted to $\Aproj$ is fully faithful.
\end{itemize}
In this case we say that $(A, f)$ is a {\sf cover} of $B$ (or that $A$ is a cover of $B$ when the functor $f$ is clear from the context).
\end{prop}

\begin{defn}[{\cite[Definition 4.37]{ROUQ}}] 
Let $i\geq 0$. We say that $(A, f)$ is an $i$-faithful cover of $B$ if
$${\rm Ext}^i_A (M, M') \cong {\rm Ext}^i_B (fM, fM')$$
for all $M, M'\in ( \Amod )^\Delta$.
\end{defn}

\begin{prop}[{\cite[Proposition 4.40]{ROUQ}}] \label{R0faithful}  The following statements are equivalent.
\begin{itemize}[leftmargin=*]
\item $(A, f)$ is a $0$-faithful cover of $B$.
\item For all $M\in (\Amod)^\Delta$, the map $\eta(M): M \rightarrow gf(M)$ is an isomorphism.
\item For all $T\in \Atilt$, the map $\eta(T): T\rightarrow gf(T)$ is an isomorphism.
\end{itemize}  
\end{prop}

\section{Combinatorics}

We now recall the combinatorics of (extended) Khovanov arc algebras (see also \cite{MR2918294} and \cite{compan2}). We let $S_n$ denote the symmetric group on $n$ letters.

\subsection{Weights and partitions}

Fix $m,n\in \mathbb{N}$. We denote by $\Lambda_{m,n}$ the set of labelled horizontal strips of length $n+m$ where each integer point $1 \leq i \leq n+m$ is labelled by either $\up$ or $\down$ in such a way that the total number of $\up$ is equal to $m$ (and so the total number of $\down$ is equal to $n$). We call the elements of $\Lambda_{m,n}$ {\sf weights}. We define the partial order $\leq$ on the set of weights to be generated by the basic operation of swapping a $\down$ and an $\up$ symbol; getting bigger means that the $\down$'s move to the right. 
    
A {\sf partition}    $\lambda $  is defined to be a weakly decreasing  sequence   of non-negative integers $\lambda = (\lambda_1, \lambda_2, \ldots )$.  We define the Young diagram of a partition to be the collection of tiles 
$$[\la]=\{[r,c] \mid 1\leq c \leq \la_r\}$$
depicted in Russian style with rows at $135^\circ$ and columns at $45^\circ$.  We identify a partition with its Young diagram.
We let $\la^t$ denote the transpose partition given by reflection 
of the Russian Young diagram through the vertical axis.  
 Given  $m,n\in \NN$ we let  ${\mathscr P}_{m,n}$ denote the set of all partitions which fit into an $m\times n$ rectangle, that is 
$${\mathscr P}_{m,n}= \{ \la \mid \la_1\leq m, \la_1^t \leq n\}.$$
There is a bijection beween $\Lambda_{m,n}$ and $\mathscr{P}_{m,n}$ given as follows. Read the labels of a weight in $\Lambda_{m,n}$ from left to right. Starting at the left most corner of the $m\times n$ rectangle, take a north-easterly step for each $\down$ and a south-easterly step for each $\up$. We end up at the rightmost corner of the rectangle, having traced out the ``northern perimeter" of the Russian Young diagram.  
An example is given in \cref{Figweightpartition}.

Throughout the paper, we will identify weights with their corresponding partitions. 
In particular we have that the maximal element in $\Lambda_{m,n}$ is given by 
$$\up \up \ldots \up \down \down \ldots \down = \varnothing$$
and the minimal element in $\Lambda_{m,n}$ is given by
$$\down \down \ldots \down \up \up \ldots \up = (m^n).$$
More generally for $\la, \mu \in \Lambda_{m,n}$ we have $\la < \mu$ if and only if the partition $\mu$ is a subset of the partition $\la$, written as $\mu \subset \la$.

\begin{rmk} The set $\Lambda_{m,n}$ is easily seen to be a labelling set for the cosets of the product of symmetric groups $S_m \times S_n$ inside $S_{m+n}$.  \cref{Figweightpartition} shows how to associate a minimal length coset representative to a given partition. Under this bijection, the partial order $\leq$ described above corresponds to the opposite of the Bruhat order.
\end{rmk}

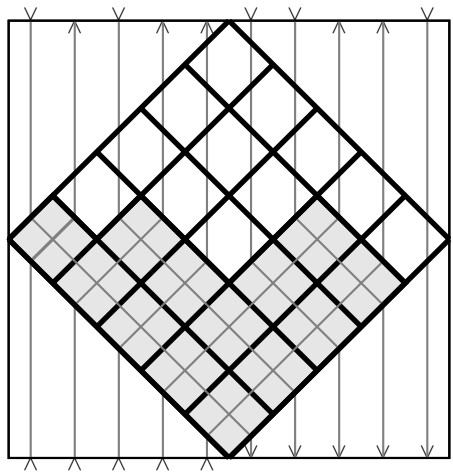
\begin{figure}[ht!]
  $$ \scalefont{0.7}
 \begin{tikzpicture} [scale=0.82]

\path (0,0) coordinate (origin2);

\begin{scope}

     \foreach \i in {0,1,2,3,4,5,6,7,8,9,10,11,12}
{
\path (origin2)--++(45:0.5*\i) coordinate (c\i);
\path (origin2)--++(135:0.5*\i)  coordinate (d\i);
  }

\path(origin2)  ++(135:2.5)   ++(-135:2.5) coordinate(corner1);
\path(origin2)  ++(45:2.5)   ++(135:7.5) coordinate(corner2);
\path(origin2)  ++(45:2.5)   ++(-45:2.5) coordinate(corner4);
\path(origin2)  ++(135:3)   ++(45:7)  ++(45:0.5)   ++(-45:0.5)  coordinate(corner3);

\draw[thick] (origin2)--(corner1)--(corner2)--(corner3)--(corner4)--(origin2);

\clip(corner1)--(corner2)--++(90:0.3)--++(0:7.5)--(corner3)--(corner4)
--++(90:-0.3)--++(180:7.5) --(corner1);

\path[name path=pathd1] (d1)--++(90:7);
 \path[name path=top] (corner2)--(corner3);
 \path [name intersections={of = pathd1 and top}];
   \coordinate (A)  at (intersection-1);
     \path(A)--++(-90:0.1) node { $\up$ };

\path[name path=pathd3] (d3)--++(90:7);
 \path[name path=top] (corner2)--(corner3);
 \path [name intersections={of = pathd3 and top}];
   \coordinate (A)  at (intersection-1);
  \path(A)--++(90:-0.1) node { $\up$ };

\path[name path=pathd5] (d5)--++(90:7);
 \path[name path=top] (corner2)--(corner3);
 \path [name intersections={of = pathd5 and top}];
   \coordinate (A)  at (intersection-1);
  \path(A)--++(90:0.1) node { $\down$ };

\path[name path=pathd7] (d7)--++(90:7);
 \path[name path=top] (corner2)--(corner3);
 \path [name intersections={of = pathd7 and top}];
   \coordinate (A)  at (intersection-1);
     \path(A)--++(90:-0.1) node { $\up$ };

\path[name path=pathd9] (d9)--++(90:7);
 \path[name path=top] (corner2)--(corner3);
 \path [name intersections={of = pathd9 and top}];
   \coordinate (A)  at (intersection-1);
    \path(A)--++(-90:-0.1) node { $\down$ };

\path[name path=pathc1] (c1)--++(90:7);
 \path[name path=top] (corner2)--(corner3);
 \path [name intersections={of = pathc1 and top}];
   \coordinate (A)  at (intersection-1);
  \path(A)--++(-90:-0.1) node { $\down$ };

\path[name path=pathc3] (c5)--++(90:7);
 \path[name path=top] (corner2)--(corner3);
 \path [name intersections={of = pathc3 and top}];
   \coordinate (A)  at (intersection-1);
    \path(A)--++(90:-0.1) node { $\up$ };

\path[name path=pathc5] (c3)--++(90:7);
 \path[name path=top] (corner2)--(corner3);
 \path [name intersections={of = pathc5 and top}];
   \coordinate (A)  at (intersection-1);
    \path(A)--++(-90:-0.1) node { $\down$ };

\path[name path=pathc5] (c9)--++(90:7);
 \path[name path=top] (corner2)--(corner3);
 \path [name intersections={of = pathc5 and top}];
   \coordinate (A)  at (intersection-1);
    \path(A)--++(-90:-0.1) node { $\down$ };

\path[name path=pathc7] (c7)--++(90:7);
 \path[name path=top] (corner2)--(corner3);
 \path [name intersections={of = pathc7 and top}];
   \coordinate (A)  at (intersection-1);
   \path(A)--++(-90:0.1) node { $\up$ };

 \path[name path=pathd1] (d1)--++(-90:7);
 \path[name path=bottom] (corner1)--(corner4);
 \path [name intersections={of = pathd1 and bottom}];
   \coordinate (A)  at (intersection-1);
   \path (A)--++(90:-0.1) node { $\up$  };

 \path[name path=pathd3] (d3)--++(-90:7);
 \path[name path=bottom] (corner1)--(corner4);
 \path [name intersections={of = pathd3 and bottom}];
   \coordinate (A)  at (intersection-1);
   \path (A)--++(90:-0.1) node { $\up$  };

  \path[name path=pathd5] (d5)--++(-90:7);
 \path[name path=bottom] (corner1)--(corner4);
 \path [name intersections={of = pathd5 and bottom}];
   \coordinate (A)  at (intersection-1);
   \path (A)--++(90:-0.1) node { $\up$  };

 \path[name path=pathd7] (d7)--++(-90:7);
 \path[name path=bottom] (corner1)--(corner4);
 \path [name intersections={of = pathd7 and bottom}];
   \coordinate (A)  at (intersection-1);
   \path (A)--++(90:-0.1) node { $\up$  };

 \path[name path=pathd9] (d9)--++(-90:7);
 \path[name path=bottom] (corner1)--(corner4);
 \path [name intersections={of = pathd9 and bottom}];
   \coordinate (A)  at (intersection-1);
   \path (A)--++(90:-0.1) node { $\up$  };

 \path[name path=pathc1] (c1)--++(-90:7);
 \path[name path=bottom] (corner1)--(corner4);
 \path [name intersections={of = pathc1 and bottom}];
   \coordinate (A)  at (intersection-1);
   \path (A)--++(90:0.1) node { $\down$  };

    \path[name path=pathc3] (c3)--++(-90:7);
 \path[name path=bottom] (corner1)--(corner4);
 \path [name intersections={of = pathc3 and bottom}];
   \coordinate (A)  at (intersection-1);
   \path (A)--++(90:0.1) node { $\down$  };

 \path[name path=pathc5] (c5)--++(-90:7);
 \path[name path=bottom] (corner1)--(corner4);
 \path [name intersections={of = pathc5 and bottom}];
   \coordinate (A)  at (intersection-1);
   \path (A)--++(90:0.1) node { $\down$  };

 \path[name path=pathc7] (c7)--++(-90:7);
 \path[name path=bottom] (corner1)--(corner4);
 \path [name intersections={of = pathc7 and bottom}];
   \coordinate (A)  at (intersection-1);
   \path (A)--++(90:0.1) node { $\down$  };

 \path[name path=pathc7] (c9)--++(-90:7);
 \path[name path=bottom] (corner1)--(corner4);
 \path [name intersections={of = pathc7 and bottom}];
   \coordinate (A)  at (intersection-1);
   \path (A)--++(90:0.1) node { $\down$  };

\clip(corner1)--(corner2)--(corner3)--(corner4)--(corner1);

  \foreach \i in {1,3,5,7,9,11}
{
 \draw[thick, gray](c\i)--++(90:7);
 \draw[thick, gray](c\i)--++(-90:7);
\draw[thick, gray](d\i)--++(90:7);
\draw[thick, gray](d\i)--++(-90:7);
   }

\end{scope}

\begin{scope}

\path (0,0) coordinate (origin2);

 \foreach \i\j in {0,2,4,6,8,10}
{
\path(origin2)--++(45:0.5*\i) coordinate (a\i);
\path (origin2)--++(135:0.5*\i)  coordinate (b\j);
\draw[ line width=2 ](a\i) --++(135:5);
\draw[ line width=2 ](b\i) --++(45:5);
}

\draw[ very thick, fill=gray!20] (0,0) --++(45:4)--++(135:2)--++(-135:2) 
 --++(135:2)--++(-135:1) --++(135:1)--++(-135:1)--(0,0) ;

\clip(0,0) --++(45:4)--++(135:2)--++(-135:2) 
 --++(135:2)--++(-135:1) --++(135:1)--++(-135:1)--(0,0) ;
 
\path (0,0) coordinate (origin2);

 \foreach \i\j in {0,1,2,3,4,5,6,7,8,9,10,11,12}
{
\path (origin2)--++(45:0.5*\i) coordinate (a\i);
\path (origin2)--++(135:0.5*\i)  coordinate (b\j);}

 \foreach \i\j in {0,2,4,6,8,10,12}
{
\draw[line width=2  ](a\i)--++(135:14);
\draw[line width=2](b\j)--++(45:14);

\path (origin2)--++(45:0.5*\i)--++(135:14) coordinate (x\i);
\path (origin2)--++(135:0.5*\i)--++(45:14)  coordinate (y\j);

   }

\draw[  thick,magenta,gray](c1) --++(135:1) coordinate (cC1);
\draw[  thick,darkgreen,gray](cC1) --++(135:1) coordinate (cC1);
\draw[  thick,orange,gray](cC1) --++(135:1) coordinate (cC1);
\draw[  thick,lime,gray](cC1) --++(135:1) coordinate (cC1);
\draw[  thick,violet,gray](cC1) --++(135:1) coordinate (cC1);

\draw[  thick,gray,gray](c3) --++(135:1) coordinate (cC1);
\draw[  thick,magenta,gray](cC1) --++(135:1) coordinate (cC1);
\draw[  thick,darkgreen,gray](cC1) --++(135:1) coordinate (cC1);
\draw[  thick,orange,gray](cC1) --++(135:1) coordinate (cC1);
\draw[  thick,lime,gray](cC1) --++(135:1) coordinate (cC1);

  \draw[  thick,cyan,gray](c5) --++(135:1) coordinate (cC1);
 \draw[  thick,gray,gray](cC1) --++(135:1) coordinate (cC1);
\draw[  thick,magenta,gray](cC1) --++(135:1) coordinate (cC1);
\draw[  thick,darkgreen,gray](cC1) --++(135:1) coordinate (cC1);
\draw[  thick,orange,gray](cC1) --++(135:1) coordinate (cC1);
\draw[  thick,lime,gray](cC1) --++(135:1) coordinate (cC1);

  \draw[  thick,pink,gray](c7) --++(135:1) coordinate (cC1);
  \draw[  thick,cyan,gray](cC1) --++(135:1) coordinate (cC1);
 \draw[  thick,gray,gray](cC1) --++(135:1) coordinate (cC1);
\draw[  thick,magenta,gray](cC1) --++(135:1) coordinate (cC1);
\draw[  thick,darkgreen,gray](cC1) --++(135:1) coordinate (cC1);
\draw[  thick,orange,gray](cC1) --++(135:1) coordinate (cC1);
\draw[  thick,lime,gray,gray](cC1) --++(135:1) coordinate (cC1);

\draw[  thick,magenta,gray](d1) --++(45:1) coordinate (x1);
\draw[  thick,gray,gray](x1) --++(45:1) coordinate (x1);
 \draw[  thick,cyan,gray](x1) --++(45:1) coordinate (x1);
 \draw[  thick,pink,gray](x1) --++(45:1) coordinate (x1);

\draw[  thick,darkgreen,gray](d3) --++(45:1) coordinate (x1);
\draw[  thick,magenta,gray](x1) --++(45:1) coordinate (x1);
\draw[  thick,gray,gray](x1) --++(45:1) coordinate (x1);
 \draw[  thick,cyan,gray](x1) --++(45:1) coordinate (x1);
 \draw[  thick,pink,gray](x1) --++(45:1) coordinate (x1);

\draw[  thick, gray](d5) --++(45:1) coordinate (x1);
\draw[  thick, gray](x1) --++(45:1) coordinate (x1);
\draw[  thick, gray](x1) --++(45:1) coordinate (x1);
\draw[  thick, gray](x1) --++(45:1) coordinate (x1);
 \draw[  thick, ,gray](x1) --++(45:1) coordinate (x1);
 \draw[  thick ,gray](x1) --++(45:1) coordinate (x1);

\draw[  thick, gray](d7) --++(45:1) coordinate (x1);
\draw[  thick, gray](x1) --++(45:1) coordinate (x1);

\path(x1) --++(45:1) coordinate (x1);
\path(x1) --++(45:1) coordinate (x1);
\path(x1) --++(45:1) coordinate (x1);
\path(x1) --++(45:1) coordinate (x1);
\path(x1) --++(45:1) coordinate (x1);

\draw[very thick, gray](d9) --++(45:1) coordinate (x1);

\end{scope}


\draw[thick] (origin2)--(corner1)--(corner2)--(corner3)--(corner4)--(origin2);
\clip(origin2)--(corner1)--(corner2)--(corner3)--(corner4)--(origin2);

  \draw[ line width =2 ]
(0,0) --++(45:4)--++(135:2)--++(-135:2) 
 --++(135:2)--++(-135:1) --++(135:1)--++(-135:1)--(0,0) ;
\end{tikzpicture}
$$\caption
{
Along the top we picture the weight $\down \up \down \up \up \down \down \up \up \down  \in \Lambda_{5,5}$,
which  corresponds to the partition $(5,4,2^2)\in \mathscr{P}_{5,5}$. To obtain the corresponding minimal length coset representative of $S_5 \times S_5$ in $S_{10}$, simply 
fill each box in the partition with an $s_i$ generator. The corresponding product gives the minimal length coset representative (which when applied to the element $\varnothing$ of $\Lambda_{5,5}$ permutes the $\up$s and $\down$s so as to arrive at the given weight).
}\label{Figweightpartition}
\end{figure}

\subsection{Cups, caps and Kazhdan--Lusztig polynomials} 
 
 The following definitions come from \cite[Section 2]{MR2918294}.
    
 \begin{defn} 
\begin{itemize}[leftmargin=*]
 \item To each weight $\lambda$ we associate a {\sf cup diagram} $\underline{\lambda}$ and a {\sf cap diagram} $\overline{\lambda}$. To construct $\underline{\lambda}$,  
repeatedly find a pair of vertices labeled $\down$ $\up$ in order from left to right that are neighbours in the sense that there are only vertices already joined by cups in between. Join these new vertices together with a cup. Then repeat the process until there are no more such $\down$ $\up$ pairs. Finally draw rays down at all the remaining $\up$ and $\down$ vertices. The cap diagram $\overline{\lambda}$ is obtained by flipping $\underline{\lambda}$ through the horizontal axis. We stress that the vertices of the cup and cap diagrams  are not labeled. 
\item Let $\lambda$ and $\mu$ be weights. We can glue $\underline{\mu}$ under $\lambda$ to obtain a new diagram $\underline{\mu}\lambda$. We say that $\underline{\mu}\lambda$ is  {\sf oriented} if (i) the vertices at the ends of each cup in $\underline{\mu}$ are labelled by exactly one $\down$  and one $\up$ in the weight $\lambda$ and (ii) it is impossible to find two rays in $\underline{\mu}$ whose top vertices are labeled $\down$ $\up$  in that order from left to right in the weight $\lambda$. 
Similarly, we obtain a new diagram $\lambda \overline{\mu}$ by gluing $\overline{\mu}$ on top of $\lambda$. We say that $\lambda \overline{\mu}$ is oriented if $\underline{\mu} \lambda$ is oriented. 
\item Let $\lambda$, $\mu$ be weights such that $\underline{\mu}\lambda$ is oriented. We set the {\sf degree} of the diagram $\underline{\mu}\lambda$ (respectively $\lambda \overline{\mu}$) to be the number of clockwise oriented cups (respectively caps) in the diagram. 
\item Let $\la, \mu ,\nu$ be weights such that $\underline{\mu}\la$ and $\la \overline{\nu}$ are oriented. Then  we form a new diagram $\underline{\mu}\la\overline{\nu}$ by gluing $\underline{\mu}$ under and $\overline{\nu}$ on top of $\la$. We set ${\rm deg}(\underline{\mu}\la \overline{\nu}) = {\rm deg}(\underline{\mu}\la)+{\rm deg}(\la \overline{\nu})$.
\end{itemize}
\end{defn}

Examples are provided in Figure \ref{figure4} and \ref{cref-it2}.

  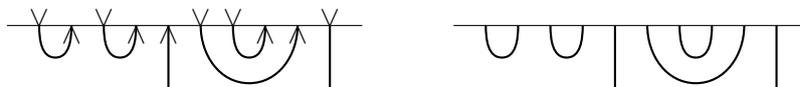
\begin{figure}[ht!]
  $$   \begin{tikzpicture} [scale=0.85]
		
		
		\path (4,1) coordinate (origin); 
		\path (origin)--++(0.5,0.5) coordinate (origin2);  
	 	\draw(origin2)--++(0:5.5); 
		\foreach \i in {1,2,3,4,5,...,10}
		{
			\path (origin2)--++(0:0.5*\i) coordinate (a\i); 
			\path (origin2)--++(0:0.5*\i)--++(-90:0.00) coordinate (c\i); 
			  }
		
		\foreach \i in {1,2,3,4,5,...,19}
		{
			\path (origin2)--++(0:0.25*\i) --++(-90:0.5) coordinate (b\i); 
			\path (origin2)--++(0:0.25*\i) --++(-90:0.9) coordinate (d\i); 
		}
		\path(a1) --++(90:0.12) node  {  $  \down   $} ;
		\path(a3) --++(90:0.12) node  {  $  \down   $} ;
		\path(a2) --++(-90:0.15) node  {  $  \up   $} ;
		\path(a4) --++(-90:0.15) node  {  $  \up   $} ;
		\path(a5) --++(-90:0.15) node  {  $  \up  $} ;
		\path(a6) --++(90:0.12) node  {  $  \down  $} ;
		\path(a7) --++(90:0.12) node  {  $  \down  $} ;
		\path(a8) --++(-90:0.15) node  {  $  \up  $} ;
		\path(a9) --++(-90:0.15) node  {  $  \up  $} ;
 \path(a10) --++(90:0.12) node  {  $  \down  $} ;

		\draw[    thick](c2) to [out=-90,in=0] (b3) to [out=180,in=-90] (c1); 
		\draw[    thick](c4) to [out=-90,in=0] (b7) to [out=180,in=-90] (c3);


		\draw[    thick](c8) to [out=-90,in=0] (b15) to [out=180,in=-90] (c7); 
 	
			\draw[    thick](c9) to [out=-90,in=0] (d15) to [out=180,in=-90] (c6);

		\draw[    thick](c5) --++(90:-1);  
		
				\draw[    thick](c10) --++(90:-1);

	\end{tikzpicture}\qquad
	  \begin{tikzpicture} [scale=0.85]
		
		
		\path (4,1) coordinate (origin); 
		\path (origin)--++(0.5,0.5) coordinate (origin2);  
	 	\draw(origin2)--++(0:5.5); 
		\foreach \i in {1,2,3,4,5,...,10}
		{
			\path (origin2)--++(0:0.5*\i) coordinate (a\i); 
			\path (origin2)--++(0:0.5*\i)--++(-90:0.00) coordinate (c\i); 
			  }
		
		\foreach \i in {1,2,3,4,5,...,19}
		{
			\path (origin2)--++(0:0.25*\i) --++(-90:0.5) coordinate (b\i); 
			\path (origin2)--++(0:0.25*\i) --++(-90:0.9) coordinate (d\i); 
		}
%
		
		\draw[    thick](c2) to [out=-90,in=0] (b3) to [out=180,in=-90] (c1); 
		\draw[    thick](c4) to [out=-90,in=0] (b7) to [out=180,in=-90] (c3);


		\draw[    thick](c8) to [out=-90,in=0] (b15) to [out=180,in=-90] (c7); 
 	
			\draw[    thick](c9) to [out=-90,in=0] (d15) to [out=180,in=-90] (c6);

		\draw[    thick](c5) --++(90:-1);  
		
				\draw[    thick](c10) --++(90:-1);

	\end{tikzpicture}$$
\caption{
Continuing with 
the weight $\la= \down \up \down \up \up \down \down \up \up \down  \in \Lambda_{5,5}$ from \cref{Figweightpartition}, we
 depict  the diagram  $\underline{\la}\la$ and the cup diagram $\underline{\la}$. 
}
\label{figure4}
	\end{figure}

\begin{defn}\cite[(5.12)]{MR2918294} 
For each $\la, \mu \in \Lambda_{m,n}$ we define the Kazhdan--Lusztig polynomial  $n_{\la \mu}(q)$ to be the monomial
$$
n_{\la,\mu}(q)= 
\begin{cases}
q^{\deg(\underline{\mu} \la)}		&\text{if $ \underline{\mu} \la $ is oriented}\\
0						&\text{otherwise.}
\end{cases}
$$
\end{defn}

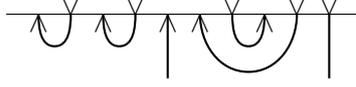
\begin{figure}[ht!]
%
%
%
%
%
%
%
%
%
%
%
%
%
%
%
%
%
%
%
%
%
%
%
%
%
%
%
%
%
$$ \begin{tikzpicture} [scale=0.85]
		
		
		\path (4,1) coordinate (origin); 
		\path (origin)--++(0.5,0.5) coordinate (origin2);  
	 	\draw(origin2)--++(0:5.5); 
		\foreach \i in {1,2,3,4,5,...,10}
		{
			\path (origin2)--++(0:0.5*\i) coordinate (a\i); 
			\path (origin2)--++(0:0.5*\i)--++(-90:0.00) coordinate (c\i); 
			  }
		
		\foreach \i in {1,2,3,4,5,...,19}
		{
			\path (origin2)--++(0:0.25*\i) --++(-90:0.5) coordinate (b\i); 
			\path (origin2)--++(0:0.25*\i) --++(-90:0.9) coordinate (d\i); 
		}
		\path(a2) --++(90:0.12) node  {  $  \down   $} ;
		\path(a4) --++(90:0.12) node  {  $  \down   $} ;
		\path(a1) --++(-90:0.15) node  {  $  \up   $} ;
		\path(a3) --++(-90:0.15) node  {  $  \up   $} ;
		\path(a5) --++(-90:0.15) node  {  $  \up  $} ;
		\path(a9) --++(90:0.12) node  {  $  \down  $} ;
		\path(a7) --++(90:0.12) node  {  $  \down  $} ;
		\path(a8) --++(-90:0.15) node  {  $  \up  $} ;
		\path(a6) --++(-90:0.15) node  {  $  \up  $} ;
 \path(a10) --++(90:0.12) node  {  $  \down  $} ;

		\draw[    thick](c2) to [out=-90,in=0] (b3) to [out=180,in=-90] (c1); 
		\draw[    thick](c4) to [out=-90,in=0] (b7) to [out=180,in=-90] (c3);


		\draw[    thick](c8) to [out=-90,in=0] (b15) to [out=180,in=-90] (c7); 
 	
			\draw[    thick](c9) to [out=-90,in=0] (d15) to [out=180,in=-90] (c6);

		\draw[    thick](c5) --++(90:-1);  
		
				\draw[    thick](c10) --++(90:-1);

	\end{tikzpicture}$$
\caption{The diagram $\underline{\mu}\la$ for 
 $\la=(4,3,1)$ and 
  $ 
\mu=  (5,4,2^2)
  $ in $\Lambda_{5,5}$. This diagram has three clockwise arcs, hence $\deg (\underline{\mu}\la) = 3$. }
\label{cref-it2}
\end{figure}

\begin{rmk}
As mentioned already, the set $\Lambda_{m,n}$ is a labelling set for the cosets of the product of symmetric groups $S_m \times S_n$ inside $S_{m+n}$. The polynomials described above are then the corresponding anti-spherical  Kazhdan--Lusztig   polynomials of type $(S_{m+n}, S_m \times S_n)$ as defined by Deodhar  for arbitrary parabolic Coxeter systems (see \cite{DD1});
the  first combinatorial description for type $(S_{m+n}, S_m \times S_n)$ was given by 
Lascoux--Sch\"{u}tzenberger  in \cite{MR646823}. 
\end{rmk}

\subsection{Regular weights}\label{regular}
A subset of the set of weights will play an important role in this paper. It can be defined in three different ways as given below.

\begin{prop} Let $m' = \min\{m,n\}$. For any $\la \in \Lambda_{m,n}$, we define 
$$\ell_t(\la) := \#\{ 1\leq j\leq t \, \text{in $\la$ labelled by $\down$}\}  \, - \,  \# \{1\leq j\leq t \, \text{in $\la$ labelled by $\up$}\}.$$
Then the following conditions are equivalent.
\begin{enumerate}
\item The partition $\la$ contains the staircase partition $(m', m'-1, m'-2, \ldots , 1)$.
\item The cup diagram $\underline{\la}$ contains $m'$ cups.
\item Assume $m'=m$. For each $1\leq t\leq m+n$ we have $\ell_t(\la) \geq 0$. 
\end{enumerate}
In this case, we say that $\la$ is {\sf regular} and we denote the set of all regular weights by $\Lambda_{m,n}^\circ$.
\end{prop}

\begin{proof}
This follows immediately from the bijection between $\Lambda_{m,n}$ and $\mathscr{P}_{m,n}$, and the construction of $\underline{\la}$.
\end{proof}

\begin{rmk} (1) There is a similar description to (3) above in the case $m'=n$ but we will not need it in this paper.\\
(2) Regular weights are called weights of maximal defect in \cite{MR2918294} and \cite{MR2600694}, where the defect of a weight $\la$ is defined to be the number of cups in $\underline{\la}$.
\end{rmk}

For each $\la\in \Lambda_{m,n}$, we will associate a regular weight $\la^\circ\in \Lambda_{m,n}^\circ$ which will play an important role in the representation theory of the  arc algebras.

\begin{defn}\label{lambdacirc} \cite[Section 6]{MR2600694}
For each $\la\in \Lambda_{m,n}$, we define the regular weight $\la^\circ\in \Lambda_{m,n}^\circ$ as follows. 
Start with the weight $\la$. Now add clockwise cups connecting $\up$ $\down$ pairs of vertices which are neighbours in the sense that there are no vertices in betwwen them not yet connected by cups. When no more such pairs are left, add nested anticlockwise cups connecting as many vertices as possible. Finally, add rays on the remaining vertices. Then $\la^\circ$ is the weight whose cup diagram $\underline{\la^\circ}$ is the one just constructed.
\end{defn}
 An example is given in Figure \ref{Figcirc}.
 
 \begin{figure}[ht!]
$$   \begin{tikzpicture} [scale=0.85]

		\path (4,1) coordinate (origin); 
 		\path (origin)--++(0.5,0.5) coordinate (origin2);  
	 	\draw(origin2)--++(0:6); 
		\foreach \i in {1,2,3,4,5,...,10,11}
		{
			\path (origin2)--++(0:0.5*\i) coordinate (a\i); 
			\path (origin2)--++(0:0.5*\i)--++(-90:0.00) coordinate (c\i); 
			  }
		
		\foreach \i in {1,2,3,4,5,...,19}
		{
			\path (origin2)--++(0:0.25*\i) --++(-90:0.5) coordinate (b\i); 
			\path (origin2)--++(0:0.25*\i) --++(-90:0.9) coordinate (d\i); 
		}
		\path(a1) --++(90:0.12) node  {  $  \down   $} ;
		\path(a3) --++(90:0.12) node  {  $  \down   $} ;
		\path(a2) --++(90:0.12) node  {  $  \down   $} ;
		\path(a4) --++(-90:0.15) node  {  $  \up   $} ;
		\path(a5) --++(-90:0.15) node  {  $  \up  $} ;
		\path(a6) --++(90:0.12) node  {  $  \down  $} ;
		\path(a7) --++(-90:0.15) node  {  $  \up   $} ;
		\path(a8) --++(-90:0.15) node  {  $  \up  $} ;
 \path(a9) --++(90:0.12) node  {  $  \down  $} ; 
 \path(a10) --++(90:0.12) node  {  $  \down  $} ; 
 				\path(a11) --++(-90:0.15) node  {  $  \up  $} ;

	\draw[    thick](c1) --++(90:-1.2) coordinate(X);  
		\draw[   densely dotted, thick](X) --++(90:-0.4) coordinate(X);  

		\draw[    thick](c4) to [out=-90,in=0] (b7) to [out=180,in=-90] (c3); 
		\draw[    thick](c6) to [out=-90,in=0] (b11) to [out=180,in=-90] (c5); 		
		\draw[    thick](c9) to [out=-90,in=0] (b17) to [out=180,in=-90] (c8); 
		\path(b17)--++(-90:0.5) coordinate (b17);
		\draw[    thick](c10) to [out=-90,in=0] (b17) to [out=180,in=-90] (c7);

		\path(b13)--++(-90:1) coordinate (b13);

		\draw[    thick](c11) to [out=-90,in=0] (b13) to [out=180,in=-90] (c2); 				
%
%
%
%
%
%
%
%
%
		
%
\draw(4,1.5) node {$\la$};

\draw(4,0.85) node {$\phantom{_\circ}\underline{\la^\circ}$};		 
	\end{tikzpicture}
$$
\vspace{0.5cm}
$$
   \begin{tikzpicture} [scale=0.85]

		\path (4,1) coordinate (origin); 
 		\path (origin)--++(0.5,0.5) coordinate (origin2);  
	 	\draw(origin2)--++(0:6); 
		\foreach \i in {1,2,3,4,5,...,10,11}
		{
			\path (origin2)--++(0:0.5*\i) coordinate (a\i); 
			\path (origin2)--++(0:0.5*\i)--++(-90:0.00) coordinate (c\i); 
			  }
		
		\foreach \i in {1,2,3,4,5,...,19}
		{
			\path (origin2)--++(0:0.25*\i) --++(-90:0.5) coordinate (b\i); 
			\path (origin2)--++(0:0.25*\i) --++(-90:0.9) coordinate (d\i); 
		}
		\path(a1) --++(90:0.12) node  {  $  \down   $} ;
		\path(a3) --++(90:0.12) node  {  $  \down   $} ;
		\path(a2) --++(90:0.12) node  {  $  \down   $} ;
		\path(a4) --++(-90:0.15) node  {  $  \up   $} ;
		\path(a5) --++(90:0.12) node  {  $  \down  $} ;
		\path(a6) --++(-90:0.15) node  {  $  \up  $} ;
		\path(a7) --++(90:0.12) node  {  $  \down  $} ;
		\path(a8) --++(90:0.12) node  {  $  \down  $} ; 
  				\path(a11) --++(-90:0.15) node  {  $  \up  $} ;
 				\path(a10) --++(-90:0.15) node  {  $  \up  $} ;
 				\path(a9) --++(-90:0.15) node  {  $  \up  $} ;

%
%

%
%
%
%
%
%
%
%
%
		
%

\draw(4,1.5) node {$\phantom{_\circ} {\la^\circ}$};		 
	\end{tikzpicture}	$$   \caption{An example of the construction of the regular $\la^\circ = (5^3,4,3^2) \in \Lambda_{5,6}$ associated to  $\la = (5^3,3,1^2) \in \Lambda_{5,6}$}\label{Figcirc}
 \end{figure}
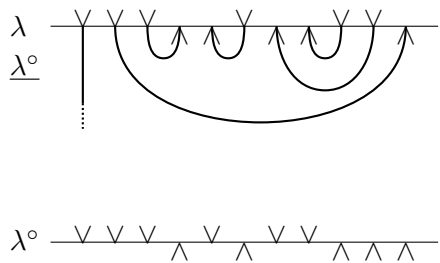

\subsection{Inverse  Kazhdan--Lusztig   polynomials}\label{SectioninverseKL}

In \cite[Section 5]{MR2600694}, Brundan and Stroppel give a closed combinatorial description of another family of polynomials $p_{\la \mu}(q)$. These polynomials can be defined by specifying that the matrix $(p_{\la \mu}(-q))$ is the inverse of the matrix of  anti-spherical Kazhdan--Lusztig   polynomials $(n_{\la\mu}(q))$ (see \cite[Corollary 5.4]{MR2600694}). We will not need their explicit description but will instead recall the recursive definition of these and deduce some useful properties. We start by introducing some notation which will be useful throughout the paper.

\medskip

\begin{defn}\label{plusminusprime}  For $1\leq i< m+n$ we will write $\Lambda_{m,n}^{\down \up}(i)$ for the subset of all weights $\la \in \Lambda_{m,n}$ with vertices $i$ and $i+1$ labelled by $\down$ and $\up$ respectively. Similarly, we will write $\Lambda_{m,n}^{\up \down}(i)$ for the subset of all weights $\la\in \Lambda_{m,n}$ with vertices $i$ and $i+1$ labelled by $\up$ and $\down$ respectively. Now fix some $1\leq i < m+n$.  
Given any $\la\in \Lambda_{m,n}^{\down \up}(i)\cup \Lambda_{m,n}^{\up \down}(i)$, we will denote by $\la'$ the unique weight in $\Lambda_{m-1,n-1}$ obtained from $\la$ by removing the symbols in position $i$ and $i+1$. 
Conversely, given any $\la'\in \Lambda_{m-1,n-1}$, we will denote by $\la^+$, respectively $\la^-$, 
 the unique weight  in $\Lambda_{m,n}^{\down \up}(i)$, respectively $\Lambda_{m,n}^{\up \down}(i)$, obtained from $\la'$ by inserting $\down \up$, respectively $\up \down$, between the first
  $(i-1)$ symbols and the last $(m+n-i-1)$ symbols.
Examples are given in Figure \ref{Figplusminus}.
\end{defn}

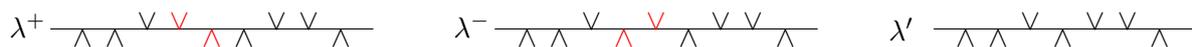
\begin{figure}[ht!]
$$  
\hspace{-0.3cm}   \begin{tikzpicture} [scale=0.85]
	 		\path (4,1) coordinate (origin); 
 		\path (origin)--++(0.5,0.5) coordinate (origin2);  
	 	\draw(origin2)--++(0:5); 
		\foreach \i in {1,2,3,4,5,...,10,11}
		{
			\path (origin2)--++(0:0.5*\i) coordinate (a\i); 
			\path (origin2)--++(0:0.5*\i)--++(-90:0.00) coordinate (c\i); 
			  }
		
		\foreach \i in {1,2,3,4,5,...,19}
		{
			\path (origin2)--++(0:0.25*\i) --++(-90:0.5) coordinate (b\i); 
			\path (origin2)--++(0:0.25*\i) --++(-90:0.9) coordinate (d\i); 
		}
		\path(a3) --++(90:0.12) node  {  $  \down   $} ;
		\path(a4) --++(90:0.12) node  {\color{red}  $  \down   $} ; 
		\path(a7) --++(90:0.12) node  {  $  \down   $} ;		
		\path(a8) --++(90:0.12) node  {  $  \down   $} ;		
		
		\path(a1) --++(-90:0.15) node  {  $  \up   $} ;
		\path(a2) --++(-90:0.15) node  {  $  \up   $} ;
		\path(a5) --++(-90:0.15) node  { \color{red} $  \up   $} ;
		\path(a6) --++(-90:0.15) node  {  $  \up   $} ;
		\path(a9) --++(-90:0.15) node  {  $  \up   $} ;
		
\draw(4,1.5) node {$\phantom{_+} {\la^+}$};		 
	\end{tikzpicture}	
	\quad\!\!
	   \begin{tikzpicture} [scale=0.85]
	 		\path (4,1) coordinate (origin); 
 		\path (origin)--++(0.5,0.5) coordinate (origin2);  
	 	\draw(origin2)--++(0:5); 
		\foreach \i in {1,2,3,4,5,...,10,11}
		{
			\path (origin2)--++(0:0.5*\i) coordinate (a\i); 
			\path (origin2)--++(0:0.5*\i)--++(-90:0.00) coordinate (c\i); 
			  }
		
		\foreach \i in {1,2,3,4,5,...,19}
		{
			\path (origin2)--++(0:0.25*\i) --++(-90:0.5) coordinate (b\i); 
			\path (origin2)--++(0:0.25*\i) --++(-90:0.9) coordinate (d\i); 
		}
		\path(a3) --++(90:0.12) node  {  $  \down   $} ;
		\path(a5) --++(90:0.12) node  {\color{red}  $  \down   $} ; 
		\path(a7) --++(90:0.12) node  {  $  \down   $} ;		
		\path(a8) --++(90:0.12) node  {  $  \down   $} ;		
		
		\path(a1) --++(-90:0.15) node  {  $  \up   $} ;
		\path(a2) --++(-90:0.15) node  {  $  \up   $} ;
		\path(a4) --++(-90:0.15) node  { \color{red} $  \up   $} ;
		\path(a6) --++(-90:0.15) node  {  $  \up   $} ;
		\path(a9) --++(-90:0.15) node  {  $  \up   $} ;
		
\draw(4,1.5) node {$\phantom{_-} {\la^-}$};		 
	\end{tikzpicture}	
	\quad
	   \begin{tikzpicture} [scale=0.85]
	 		\path (4,1) coordinate (origin); 
 		\path (origin)--++(0.5,0.5) coordinate (origin2);  
	 	\draw(origin2)--++(0:4); 
		\foreach \i in {1,2,3,4,5,...,10,11}
		{
			\path (origin2)--++(0:0.5*\i) coordinate (a\i); 
			\path (origin2)--++(0:0.5*\i)--++(-90:0.00) coordinate (c\i); 
			  }
		
		\foreach \i in {1,2,3,4,5,...,19}
		{
			\path (origin2)--++(0:0.25*\i) --++(-90:0.5) coordinate (b\i); 
			\path (origin2)--++(0:0.25*\i) --++(-90:0.9) coordinate (d\i); 
		}
		\path(a3) --++(90:0.12) node  {  $  \down   $} ;
		\path(a5) --++(90:0.12) node  {  $  \down   $} ;		
		\path(a6) --++(90:0.12) node  {  $  \down   $} ;		
		
		\path(a1) --++(-90:0.15) node  {  $  \up   $} ;
		\path(a2) --++(-90:0.15) node  {  $  \up   $} ;
		\path(a4) --++(-90:0.15) node  {  $  \up   $} ;
		\path(a7) --++(-90:0.15) node  {  $  \up   $} ;
		
\draw(4,1.5) node {$ {\la'}$};		 
	\end{tikzpicture}	$$  \caption{Examples of the weights $\la'\in \Lambda_{4,3}$, $\la^+\in \Lambda_{5,4}^{\down \up}(4)$ and $\la^-\in\Lambda_{5,4}^{\up \down}(4)$. We have not drawn $\la$ explicitly here, but notice that $\la \in \{\la^+, \la^-\}$ by definition. }\label{Figplusminus}
\end{figure}

\begin{defn} \cite[Lemma 5.2]{MR2600694}
For each $\la, \mu\in \Lambda_{m,n}$, we define the polynomial $p_{\la \mu}(q)\in \mathbb{N}_0[q]$ inductively as follows.
We set $p_{\la \la} = 1$ and $p_{\la \mu}=0$ if $\mu \ngeq \la$. Now if $\mu >\la$, we can find $i$ such that $\la = \la^+ \in \Lambda^{\down \up}_{m,n}(i)$ and we set
$$ p_{\la \mu}(q) = p_{\la^+ \mu}(q) = \left\{ \begin{array}{ll} p_{\la' \mu'}(q) + q p_{\la^- \mu} & \mbox{if $\mu = \mu^+\in \Lambda_{m,n}^{\down \up}(i)
$}\\ 
qp_{\la^- \mu}(q) & \mbox{otherwise.} \end{array}\right.$$
\end{defn}

We will only require some properties which follow easily from the inductive description given above. We start by making the following definition.

\begin{defn}
For $\la ,\mu\in \Lambda_{m,n}$ we write $\la \rightarrow \mu$ if $\mu$ can be obtained from $\la$ by swapping a symbol $\down$ in position $i$ with a symbol $\up$ in position $j>i$  such that the sequence of symbols in between, that is  in position $i+1, i+2, \ldots , j-1$, belongs to $\Lambda^\circ_{t,t}$ with $2t=j-i-1\geq 0$.
\end{defn}

\begin{prop}\label{inverseKL}
Assume $m\leq n$. Then we have 
\begin{align}\label{inverseKL1}
p_{(m^n)(m^m)}(q) = q^{m(n-m)}.
\end{align}
  Write $p_{\la \mu}(q) = \sum_{k\geq 0}p^{(k)}_{\la \mu} q^k$ with $p^{(k)}_{\la \mu}\in \mathbb{N}_0$. 
Then if $p_{\la \mu}^{(k)} \neq 0$ then we have
\begin{align}\label{inverseKL2}
\la = \la_{0} \rightarrow \la_{1} \rightarrow \la_{2} \rightarrow \ldots \rightarrow \la_{k} = \mu
\end{align}
for some $\la_{i}\in \Lambda_{m,n}$.
\end{prop}

\begin{proof}
We first consider \eqref{inverseKL1}. Recall that $(m^n) = \down \ldots \down \up \ldots \up$ ($n$ $\down$'s followed by $m$ $\up$'s) and $(m^m) = \down \ldots \down \up \ldots \up \down \ldots \down$ ($m$ $\down$'s followed by $m$ $\up$'s followed by $n-m$ $\down$'s). So, by repeatedly applying the rule $p_{\la (m^m)}(q) = qp_{\la^-(m^m)}(q)$ we can move all the $\up$ symbols in the weight $(m^n)$ to the left, past the last $n-m$ $\down$ symbols and hence obtain
$$p_{(m^n)(m^m)} (q)= q^{m(n-m)}p_{(m^m)(m^m)}(q) = q^{m(n-m)}.$$
(Note that we have $m$ $\up$ symbols moving past $n-m$ $\down$ symbols so we apply the rule a total of  $m(n-m)$ times.)

We now consider \eqref{inverseKL2}.  We proceed by induction on $m\in\ZZ_{\geq0}$. If $m=0$ there is nothing to prove. Now assume that the result holds for $m-1$. We will use downwards induction on $\la$. If $\la$ is maximal then $p_{\la \mu}(q)\neq 0$ implies that $\mu = \lambda$ and there is nothing to prove. Now take $\la$ non-maximal and assume that the result holds for all weights greater than $\lambda$. Let $\mu\in \Lambda_{m,n}$ with $p_{\la \mu}(q)\neq 0$. If $\mu =\la$, then again  there is nothing to prove.  So assume $\mu >\la$ then we have $\la = \la^+\in \Lambda^{\down \up}_{m,n}(i)$ for some $i\geq1$ and we have
$$p_{\la \mu}(q) = p_{\la^+ \mu}(q) = \left\{ \begin{array}{ll} p_{\la' \mu'}(q) + q p_{\la^- \mu} & \mbox{if $\mu = \mu^+\in \Lambda_{m,n}^{\down \up}(i)$}\\ 
q p_{\la^- \mu}(q) & \mbox{otherwise.} \end{array}\right.$$
By induction, as $\la'\in \Lambda_{m-1,n-1}$, we have that if $p_{\la' \mu'}^{(k)} \neq 0$ then we have 
$$\la' = \la'_{0} \rightarrow \la'_{1}  \rightarrow \ldots \rightarrow \la'_{k} = \mu'.$$
Now, note that for each $\la'_{t}\in \Lambda_{m-1,n-1}$, we have that the corresponding $\la^+_t$ is obtained from $\la'_t$ by adding a pair of neighbouring symbols $\down \up$ in positions $i$ and $i+1$. In particular, if $\la'_t \rightarrow \la'_{t+1}$ then we have $\la^+_t \rightarrow \la^+_{t+1}$. Thus we obtain
$$\la = \la^+ = \la^+_{0} \rightarrow \la^+_{1}  \rightarrow \ldots \rightarrow \la^+_{k} = \mu^+$$
as required. 
Finally, assume $p_{\la^- \mu}^{(k-1)}\neq 0$. Note that $\la^- > \la = \la^+$, so by induction we have 
$$\la^- = \la_0 \rightarrow \la_1 \rightarrow \ldots \rightarrow \la_{k-1} = \mu.$$
Now as $\la = \la^+ \rightarrow \la^-$ ($\la^-$ is obtained from $\la^+$ by swapping the symbols $\down \up$ in position $i$ and $i+1$), we obtain  the required sequence.
\end{proof}

We will need the following combinatorial lemma in Section 8.

\begin{lem}\label{littleclaim}
Let $\la , \mu\in \Lambda_{m,n}$ and assume that $\la \rightarrow \mu$. Then we have 
$$\min \{ \ell_h(\mu) \, : \, \mbox{$h$ is labelled by $\up$ in $\mu$}\} \geq \min \{\ell_h(\la) \, : \, \mbox{$h$ is labelled by $\up$ in $\la$}\} \, - \, 1.$$
\end{lem}

\begin{proof} Assume that $\mu$ is obtained from $\la$ by swapping the symbols $\down$ $\up$ in positions $i<j$ respectively.
By definition of $\la \rightarrow \mu$ and the $\ell_h$ function we have
$$\ell_i (\mu) = \ell_j(\la) -1$$
and 
$$\ell_k(\mu) = \left\{ \begin{array}{ll} \ell_k(\la) - 2 & \mbox{for $i<k<j$}\\ \ell_k(\la) & \mbox{for $k\geq j$ or $k<i$.}\end{array}\right.$$
Now note that, as the symbols in positions $i+1, \ldots , j-1$ in $\la$ (and $\mu$) form a weight in $\Lambda^\circ_{t,t}$ for some $t\geq 0$ we must have $\ell_k(\la) \geq \ell_j(\la) + 1$ for all $i<k<j$. This implies that for any $i<k<j$ with $k$ labelled by $\up$ we have 
$$\ell_k(\la) \geq \min \{\ell_h(\la) \, : \, \mbox{$h$ is labelled by $\up$ in $\la$}\} + 1.$$
This proves the claim.
\end{proof}

\section{The extended Khovanov arc algebras $K^m_n$}

We now introduce the quasi-hereditary algebras of interest in this paper, the      extended Khovanov arc algebras, 
and recall some of the results  concerning  their representation theory from \cite{MR2600694,MR2918294} which will be useful in what follows.
\subsection{Definition}

 We now recall the definition of  the extended Khovanov arc algebras studied in  \cite{MR2600694,MR2781018,MR2955190,MR2881300}. 
Let $\Bbbk$ be a field. We define $K^m_n$ to be the $\Bbbk$-algebra spanned by the set of diagrams 
$$
\{
\underline{\la}
\mu \overline{\nu}
\mid \la,\mu,\nu \in \Lambda_{m,n} \text{ such that }
\mu\overline{\nu},  \underline{\la}
\mu \text{  are oriented}\}
$$
 with the multiplication defined as follows.
First set 
 $$(\underline{\la}
\mu \overline{\nu})(\underline{\alpha}
\beta \overline{\gamma}) = 0 \quad \mbox{unless $\nu = \alpha$}.$$ 
To compute $(\underline{\la}
\mu \overline{\nu})(\underline{\nu}
\beta \overline{\gamma})$ place $(\underline{\la}
\mu \overline{\nu})$ under $(\underline{\nu}
\beta \overline{\gamma})$ then follow the \lq surgery' procedure.
This surgery combines two circles into one or splits one circle into two  using the following rules for re-orientation (where we use the notation
$1=\text{anti-clockwise circle}$, $x=\text{clockwise circle}$, $y=\text{oriented strand}$).  We have the splitting rules 
$$1 \mapsto 1 \otimes x + x \otimes 1,
\quad
 x \mapsto x \otimes x,
 \quad
 y \mapsto x \otimes y. 
 $$ 
 and the merging rules 
\begin{align*}
1 \otimes 1 \mapsto 1, 
\quad
1 \otimes x \mapsto x,
\quad
 x \otimes 1 \mapsto x,
\quad
x \otimes x \mapsto 0,
\quad
1 \otimes y  \mapsto y,\quad
x \otimes y \mapsto 0,
\end{align*}
\begin{align*}
y \otimes y  \mapsto
\left\{
\begin{array}{ll}
y\otimes y&\text{if both strands are propagating,  
one   is
 }\\
&\text{$\up$-oriented and  the other   is $\down$-oriented;}\\
0&\text{otherwise.} \end{array}
\right.
\end{align*}


\begin{eg}\label{brace-surgery}
We have the following product of Khovanov diagrams

$$   \begin{minipage}{3.1cm}
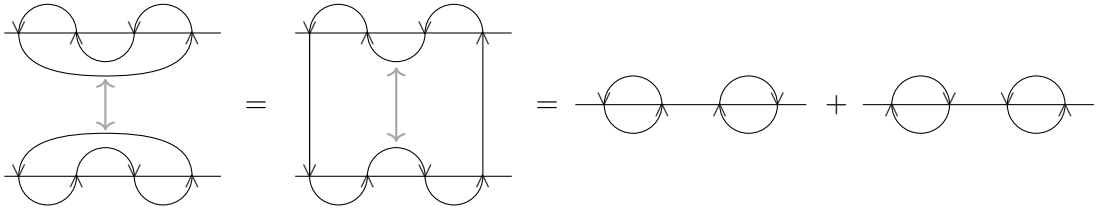
 \begin{tikzpicture}  [scale=0.76]

 \draw(-0.25,0)--++(0:3.75) coordinate (X);

   \draw[thick](0 , 0.09) node {$\scriptstyle\down$};        
   \draw(2,0.09) node {$\scriptstyle\down$};
    \draw[thick](1 ,-0.09) node {$\scriptstyle\up$};
   \draw[thick](3 ,-0.09) node {$\scriptstyle\up$};

    \draw
    (0,0)
      to [out=90,in=180] (0.5,0.5)  to [out=0,in=90] (1,0)
        (1,0)
 to [out=-90,in=180] (1.5, -0.5) to [out=0,in= -90] (2,0)
  to [out=90,in=180] (2.5,0.5)  to [out=0,in=90] (3,0)  
   to [out=-90,in=0] (1.5, -0.75) to [out=180,in= -90] (0,0)
  ;

\draw[gray!70,<->,thick](1.5,-0.8)-- (1.5,-2.5+0.8);

   \draw(-0.25,-2.5)--++(0:3.75) coordinate (X);

   \draw[thick](0 , 0.09-2.5) node {$\scriptstyle\down$};        
   \draw(2,0.09-2.5) node {$\scriptstyle\down$};

   \draw[thick](1 ,-0.09-2.5) node {$\scriptstyle\up$};
   \draw[thick](3 ,-0.09-2.5) node {$\scriptstyle\up$};

    \draw
    (0,-2.5)
      to [out=-90,in=180] (0.5,-2.5-0.5)  to [out=0,in=-90] (1,-2.5-0)
        (1,-2.5-0)
 to [out=90,in=180] (1.5, -2.5+0.5) to [out=0,in= 90] (2,-2.5-0)
  to [out=-90,in=180] (2.5,-2.5-0.5)  to [out=0,in=-90] (3,-2.5-0)  
   to [out=90,in=0] (1.5, -2.5+0.75) to [out=180,in= 90] (0,-2.5-0)
  ;

 \end{tikzpicture} \end{minipage}
 = \; \;
 \begin{minipage}{3.1cm}
 \begin{tikzpicture}  [scale=0.76]

 \draw(-0.25,0)--++(0:3.75) coordinate (X);

   \draw[thick](0 , 0.09) node {$\scriptstyle\down$};        
   \draw(2,0.09) node {$\scriptstyle\down$};
    \draw[thick](1 ,-0.09) node {$\scriptstyle\up$};
   \draw[thick](3 ,-0.09) node {$\scriptstyle\up$};

    \draw
    (0,0)  --++(-90:2.5)  --++( 90:2.5)
      to [out=90,in=180] (0.5,0.5)  to [out=0,in=90] (1,0)
        (1,0)
 to [out=-90,in=180] (1.5, -0.5) to [out=0,in= -90] (2,0)
  to [out=90,in=180] (2.5,0.5)  to [out=0,in=90] (3,0)  
 
   ;

\draw[gray!70,<->,thick](1.5,-0.6)-- (1.5,-2.5+0.6);

   \draw(-0.25,-2.5)--++(0:3.75) coordinate (X);

   \draw[thick](0 , 0.09-2.5) node {$\scriptstyle\down$};        
   \draw(2,0.09-2.5) node {$\scriptstyle\down$};

   \draw[thick](1 ,-0.09-2.5) node {$\scriptstyle\up$};
   \draw[thick](3 ,-0.09-2.5) node {$\scriptstyle\up$};

    \draw
    (0,-2.5)
      to [out=-90,in=180] (0.5,-2.5-0.5)  to [out=0,in=-90] (1,-2.5-0)
        (1,-2.5-0)
 to [out=90,in=180] (1.5, -2.5+0.5) to [out=0,in= 90] (2,-2.5-0)
  to [out=-90,in=180] (2.5,-2.5-0.5)  to [out=0,in=-90] (3,-2.5-0)--++(90:2.5)  
   ;

 \end{tikzpicture} \end{minipage}
 =\;
 \begin{minipage}{3.1cm}
 \begin{tikzpicture}  [scale=0.76]

 \draw(-0.5,0)--++(0:4) coordinate (X);

   \draw[thick](0 , 0.09) node {$\scriptstyle\down$};        
    \draw[thick](1 ,-0.09) node {$\scriptstyle\up$};

      \draw[thick](2 ,-0.09) node {$\scriptstyle\up$};  
   \draw(3,0.09) node {$\scriptstyle\down$};

    \draw
    (0,0)
      to [out=90,in=180] (0.5,0.5)  to [out=0,in=90] (1,0)
     to [out=-90,in=0] (0.5,-0.5)  to [out=180,in=-90] (0,0)
         
    ;

   \draw
    (2+0,0)
      to [out=90,in=180] (2+0.5,0.5)  to [out=0,in=90] (2+1,0)
     to [out=-90,in=0] (2+0.5,-0.5)  to [out=180,in=-90] (2+0,0)
         
    ;

 \end{tikzpicture} \end{minipage} 
 \;+\;
  \begin{minipage}{3.4cm}
 \begin{tikzpicture}  [scale=0.76]

 \draw(-0.5,0)--++(0:4) coordinate (X);

   \draw[thick](1 , 0.09) node {$\scriptstyle\down$};        
    \draw[thick](0 ,-0.09) node {$\scriptstyle\up$};

      \draw[thick](3 ,-0.09) node {$\scriptstyle\up$};  
   \draw(2,0.09) node {$\scriptstyle\down$};

    \draw
    (0,0)
      to [out=90,in=180] (0.5,0.5)  to [out=0,in=90] (1,0)
     to [out=-90,in=0] (0.5,-0.5)  to [out=180,in=-90] (0,0)
         
    ;

   \draw
    (2+0,0)
      to [out=90,in=180] (2+0.5,0.5)  to [out=0,in=90] (2+1,0)
     to [out=-90,in=0] (2+0.5,-0.5)  to [out=180,in=-90] (2+0,0)
         
    ;

 \end{tikzpicture} \end{minipage}  $$ 
 where we highlight with arrows the pair of arcs on which we are about to perform surgery.  The first equality follows from   the merging rule for $1\otimes 1 \mapsto 1$ and the second equality follows from  the splitting rule 
$ 1 \mapsto 1 \otimes x + x \otimes 1$. 
 \end{eg}

\begin{eg}\label{brace-surgery2}
We have the following product of Khovanov diagrams

$$   \begin{minipage}{3.1cm}
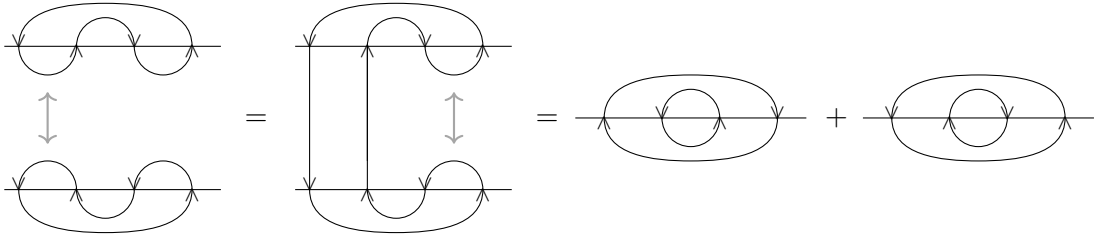
 \begin{tikzpicture}  [scale=0.76]

 \draw(-0.25,0)--++(0:3.75) coordinate (X);

   \draw[thick](0 , 0.09) node {$\scriptstyle\down$};        
   \draw(2,0.09) node {$\scriptstyle\down$};
    \draw[thick](1 ,-0.09) node {$\scriptstyle\up$};
   \draw[thick](3 ,-0.09) node {$\scriptstyle\up$};

    \draw
    (0,0)
      to [out=-90,in=180] (0.5,-0.5)  to [out=0,in=-90] (1,0)
        (1,0)
 to [out=90,in=180] (1.5, 0.5) to [out=0,in= 90] (2,0)
  to [out=-90,in=180] (2.5,-0.5)  to [out=0,in=-90] (3,0)  
   to [out=90,in=0] (1.5, 0.75) to [out=180,in= 90] (0,0)
  ;

\draw[gray!70,<->,thick](0.5,-0.8)-- (0.5,-2.5+0.8);

   \draw(-0.25,-2.5)--++(0:3.75) coordinate (X);

   \draw[thick](0 , 0.09-2.5) node {$\scriptstyle\down$};        
   \draw(2,0.09-2.5) node {$\scriptstyle\down$};

   \draw[thick](1 ,-0.09-2.5) node {$\scriptstyle\up$};
   \draw[thick](3 ,-0.09-2.5) node {$\scriptstyle\up$};

    \draw
    (0,-2.5)
      to [out=90,in=180] (0.5,-2.5+0.5)  to [out=0,in=90] (1,-2.5-0)
        (1,-2.5-0)
 to [out=-90,in=180] (1.5, -2.5-0.5) to [out=0,in= -90] (2,-2.5-0)
  to [out=90,in=180] (2.5,-2.5+0.5)  to [out=0,in=90] (3,-2.5-0)  
   to [out=-90,in=0] (1.5, -2.5-0.75) to [out=180,in= -90] (0,-2.5-0)
  ;

 \end{tikzpicture} \end{minipage}
 = \; \;
%
%
%
%
%
%
%
%
%
%
%
%
%
%
%
%
%
%
%
%
%
%
%
%
%
%
%
%
%
%
%
%
%
%
%
%
 \begin{minipage}{3.1cm}
 \begin{tikzpicture}  [scale=0.76]

 \draw(-0.25,0)--++(0:3.75) coordinate (X);

   \draw[thick](0 , 0.09) node {$\scriptstyle\down$};        
   \draw(2,0.09) node {$\scriptstyle\down$};
    \draw[thick](1 ,-0.09) node {$\scriptstyle\up$};
   \draw[thick](3 ,-0.09) node {$\scriptstyle\up$};

    \draw
 (1,-2)--(1,0)
        (1,0)
 to [out=90,in=180] (1.5, 0.5) to [out=0,in= 90] (2,0)
  to [out=-90,in=180] (2.5,-0.5)  to [out=0,in=-90] (3,0)  
   to [out=90,in=0] (1.5, 0.75) to [out=180,in= 90] (0,0)
  ;

\draw[gray!70,<->,thick](2.5,-0.8)-- (2.5,-2.5+0.8);

   \draw(-0.25,-2.5)--++(0:3.75) coordinate (X);

   \draw[thick](0 , 0.09-2.5) node {$\scriptstyle\down$};        
   \draw(2,0.09-2.5) node {$\scriptstyle\down$};

   \draw[thick](1 ,-0.09-2.5) node {$\scriptstyle\up$};
   \draw[thick](3 ,-0.09-2.5) node {$\scriptstyle\up$};

    \draw
    (0,0)--(0,-2.5);
   
    \draw
    (1,0)--
        (1,-2.5-0)
 to [out=-90,in=180] (1.5, -2.5-0.5) to [out=0,in= -90] (2,-2.5-0)
  to [out=90,in=180] (2.5,-2.5+0.5)  to [out=0,in=90] (3,-2.5-0)  
   to [out=-90,in=0] (1.5, -2.5-0.75) to [out=180,in= -90] (0,-2.5-0)
  ;

 \end{tikzpicture} \end{minipage}
  =\;
 \begin{minipage}{3.1cm}
 \begin{tikzpicture}  [scale=0.76]

 \draw(-0.5,0)--++(0:4) coordinate (X);

   \draw[thick](1 , 0.09) node {$\scriptstyle\down$};        
    \draw[thick](0 ,-0.09) node {$\scriptstyle\up$};

      \draw[thick](2 ,-0.09) node {$\scriptstyle\up$};  
   \draw(3,0.09) node {$\scriptstyle\down$};

    \draw
    (0+1,0)
      to [out=90,in=180] (0.5+1,0.5)  to [out=0,in=90] (1+1,0)
     to [out=-90,in=0] (0.5+1,-0.5)  to [out=180,in=-90] (0+1,0)
         
    ;

   \draw
    (0 ,0)
      to [out=90,in=180] (0.5+1,0.75)  to [out=0,in=90] (3,0)
     to [out=-90,in=0] (0.5+1,-0.75)  to [out=180,in=-90] (0,0)
         
    ;

%

 \end{tikzpicture} \end{minipage} 
 \;+\;
  \begin{minipage}{3.4cm}
 \begin{tikzpicture}  [scale=0.76]

 \draw(-0.5,0)--++(0:4) coordinate (X);

   \draw[thick](0 , 0.09) node {$\scriptstyle\down$};        
    \draw[thick](1 ,-0.09) node {$\scriptstyle\up$};

      \draw[thick](3 ,-0.09) node {$\scriptstyle\up$};  
   \draw(2,0.09) node {$\scriptstyle\down$};

  \draw
    (0+1,0)
      to [out=90,in=180] (0.5+1,0.5)  to [out=0,in=90] (1+1,0)
     to [out=-90,in=0] (0.5+1,-0.5)  to [out=180,in=-90] (0+1,0)
         
    ;

   \draw
    (0 ,0)
      to [out=90,in=180] (0.5+1,0.75)  to [out=0,in=90] (3,0)
     to [out=-90,in=0] (0.5+1,-0.75)  to [out=180,in=-90] (0,0)
         
    ;
   
%
%
%
%
%
%
%

 \end{tikzpicture} \end{minipage}  $$ 
 where we highlight with arrows the pair of arcs on which we are about to perform surgery. 
This is similar to \cref{brace-surgery}. 
 \end{eg}

\begin{prop}[{\cite[Section 3]{MR2918294}}] The map $^* : K^m_n \rightarrow K^m_n$ defined by $\underline{\la} \mu \overline{\nu} \mapsto (\underline{\la} \mu \overline{\nu})^* =  \underline{\nu} \mu \overline{\la}$ is an algebra anti-automorphism, giving a duality functor
$$^{\ostar} \, : \, \Kmod \rightarrow \Kmod.$$
\end{prop}

 \begin{rmk}\label{rmkmn1} There is also another algebra anti-isomorphism $\curvearrowleft \, : \, K^m_n \rightarrow K^n_m$ defined by rotating each diagram $\underline{\la}\mu\overline{\nu}$ by 180 degrees. Composing this map with the algebra anti-automorphism $^*$ gives an isomorphism between $K^m_n$ and $K^n_m$.
 \end{rmk}

\subsection{Standard modules}

Brundan and Stroppel showed in \cite[Section 5]{MR2918294} that $K_n^m$ is a quasi-hereditary algebra with respect  to $(\Lambda_{m,n}, \leq )$. For each $\la \in \Lambda_{m,n}$ we denote by $L(\la) = L_{m,n}(\la)$ the corresponding simple module, by $P(\la) = P_{m,n}(\la)$ its projective cover and by $\Delta(\la) = \Delta_{m,n}(\la)$ the corresponding standard module. So for any $\la ,\mu\in \Lambda_{m,n}$ we have 
\begin{itemize}
\item $[\Delta(\la) : L(\mu)] \neq 0 \Rightarrow \mu \leq \la$, and 
\item $ [\Delta(\la) : L(\la)] =1.$
\end{itemize}
Moreover, Brundan and Stroppel showed that $L(\la)^{\ostar} = L(\la)$ for all $\la \in \Lambda_{m,n}$.
They also describe explicitly the composition factor multiplicities for standard modules in terms of  Kazhdan--Lusztig   polynomials.

\begin{thm}[{\cite[Theorem 5.2]{MR2918294}}] \label{decnumbers}
For all $\la, \mu \in \Lambda_{m,n}$ we have 
$$[\Delta(\la): L(\mu)] = n_{\la \mu}(1) = \left\{ \begin{array}{ll} 1 & \mbox{if $\underline{\mu}\la$ is oriented} \\ 0 & \mbox{otherwise} \end{array}\right.$$
\end{thm}

\begin{rmk}
In fact, Brundan and Stroppel showed that $K^m_n$ is also positively graded and that the actual  Kazhdan--Lusztig   polynomials $n_{\la \mu}(q)$ describe the graded decomposition numbers. We will not consider the grading explicitly  in this paper as it does not play a significant role and  as it is notationally cumbersome (but one can incorporate this in a standard fashion). 
\end{rmk}

Moreover, they obtained the following structural result on standard modules.

\begin{thm}[{\cite[Theorem 6.6 and Corollary 6.7]{MR2600694}}] \label{soclestandard} 
For each $\la \in \Lambda_{m,n}$, the standard module $\Delta(\la)$ is rigid, i.e.~its radical and socle filtration coincide and for each $k\geq 0$ we have 
$$\rad_k \Delta(\la)
= \bigoplus_{\mu} L(\mu)$$
where the sum is over all $\mu \in \Lambda_{m,n}$ with $\underline{\mu}\la$ oriented of degree $k$.
Furthermore, for each $\la \in \Lambda_{m,n}$, we have 
 $\soc \Delta(\la) = L(\la^\circ)$ 
where $\la^{\circ}$ is defined in \cref{lambdacirc}.
\end{thm}

We will make use of the following special cases.

\begin{cor}\label{uniserial}
Assume $n\geq m$. Then we have 
that $\Delta_{m,n}(\varnothing)$ is uniserial of length $m+1$ with
\begin{align}\label{uniserial1}
\rad_t \Delta_{m,n}(\varnothing) = L(t^t)   
\intertext{for $0 \leq t \leq m$ 
and 
  $\Delta_{m,n}(m^{n-m})$ is uniserial of length $m+1$ with}
 \label{uniserial2}
 \rad_t \Delta_{m,n}(m^{n-m}) = L(m^{n-m},t^t)  
 \end{align}
for $ 0 \leq t \leq m$. 
\end{cor}

\begin{proof}
This follows directly from \cref{soclestandard} as $\varnothing = \up \ldots \up \down \ldots \down$ ($m$ $\up$'s followed by $n$ $\down$'s) and $(m^{n-m}) = \down \ldots \down \up \ldots \up \down \ldots \down$ ($n-m$ $\down$'s followed by $m$ $\up$'s followed by $m$ $\down$'s). 
\end{proof}

Brundan--Stroppel also gave explicit projective resolutions for standard modules.

\begin{thm}[{\cite[Theorem  5.3]{MR2918294}}]

\label{projresol}  For any $\la \in \Lambda_{m,n}$, we have an exact sequence 
$$\cdots \rightarrow P_2 \rightarrow P_1 \rightarrow P_0 \rightarrow \Delta_{m,n}(\la) \rightarrow 0$$
where 
$$P_k = \bigoplus_{\mu\in \Lambda_{m,n}} p_{\la \mu}^{(k)} P_{m,n}(\mu)$$
and $p_{\la\mu}(q) = \sum_k p_{\la \mu}^{(k)}q^k$ are the inverse  Kazhdan--Lusztig   polynomials defined in \cref{SectioninverseKL}. 
\end{thm}

\subsection{Projective functors}

One of the key ingredients used by Brundan--Stroppel throughout \cite{MR2600694,MR2781018,MR2955190,MR2881300} is their so-called projective functors. These are given by tensoring by certain bimodules defined using \lq crossingless matchings', which generalise the geometric bimodules defined by Khovanov \cite{MR1928174}. We will only need very special cases of these functors, corresponding to the matchings given by the following diagrams
$$
  
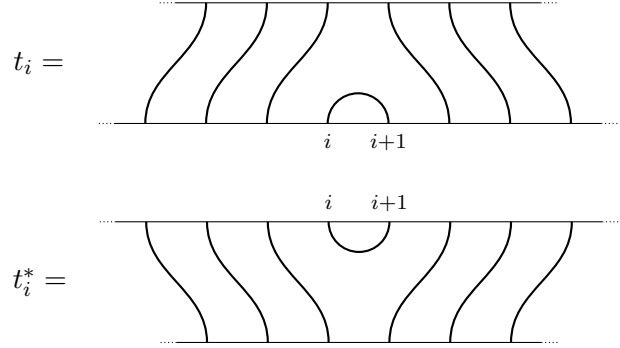
\begin{tikzpicture} [scale=0.8,yscale=-1]

		\path (4,1) coordinate (origin); 
 		\path (origin)--++(0.5,0.5) coordinate (origin2);  
	 	\draw(origin2)--++(0:6.5) coordinate(X);
			 	\draw (origin2)--++(180:1.5)coordinate (Y); 	\draw[densely dotted](Y)--++(180:0.3); 
							 	\draw[densely dotted](X)--++(0:0.3); 
		\foreach \i in {-2,-1,0,1,2,3,4,5,...,12}
		{
			\path (origin2)--++(0:0.5*\i) coordinate (a\i); 
			\path (origin2)--++(0:0.5*\i)--++(-90:0.00) coordinate (c\i); 
			  }
		
		\foreach \i in {0,1,2,3,4,5,...,19}
		{
			\path (origin2)--++(0:0.25*\i) --++(-90:0.5) coordinate (b\i); 
			\path (origin2)--++(0:0.25*\i) --++(-90:0.9) coordinate (d\i); 
		}

	\path (4,-1) coordinate (origin); 
 		\path (origin)--++(0.5,0.5) coordinate (origin2);  
	 	\draw(origin2)--++(0:5.5) coordinate(X);
			 	\draw (origin2)--++(180:0.5)coordinate (Y); 	\draw[densely dotted](Y)--++(180:0.3); 
							 	\draw[densely dotted](X)--++(0:0.3); 
		\foreach \i in {0,1,2,3,4,5,...,10}
		{
			\path (origin2)--++(0:0.5*\i) coordinate (Da\i); 
			\path (origin2)--++(0:0.5*\i)--++(-90:0.00) coordinate (Dc\i); 
			  }
		
		\foreach \i in {1,2,3,4,5,...,19}
		{
			\path (origin2)--++(0:0.25*\i) --++(-90:0.5) coordinate (Db\i); 
			\path (origin2)--++(0:0.25*\i) --++(-90:0.9) coordinate (Dd\i); 
		}
		
				\path(Dc4) node [above] {\scalefont{0.7}$\phantom{{+}1}\phantom{{+}1}$};




		\draw[    thick](c8) to [out=-90, in=90] (Dc6);
		\draw[    thick](c10) to [out=-90, in=90] (Dc8);
		\draw[    thick](c12) to [out=-90, in=90] (Dc10);

		\draw[    thick](c-2) to [out=-90, in=90] (Dc0);

		\draw[    thick](c0) to [out=-90, in=90] (Dc2);
		\draw[    thick](c2) to [out=-90, in=90] (Dc4);
		\draw[    thick](c6) to [out=-90,in=0] (b10) to [out=180,in=-90] (c4); 
		\draw(c6) node [below] {\scalefont{0.7}$i{+}1$};
		\draw(c4) node [below] {\scalefont{0.7}$\phantom{{+}1}i\phantom{{+}1}$};		

%
%
%
%
%
%
%
	\draw(1.75,0.5)	 node {$t_i=$};
	\end{tikzpicture}	
	$$
	 $$
  \begin{tikzpicture} [scale=0.8]

		\path (4,1) coordinate (origin); 
 		\path (origin)--++(0.5,0.5) coordinate (origin2);  
	 	\draw(origin2)--++(0:6.5) coordinate(X);
			 	\draw (origin2)--++(180:1.5)coordinate (Y); 	\draw[densely dotted](Y)--++(180:0.3); 
							 	\draw[densely dotted](X)--++(0:0.3); 
		\foreach \i in {-2,-1,0,1,2,3,4,5,...,12}
		{
			\path (origin2)--++(0:0.5*\i) coordinate (a\i); 
			\path (origin2)--++(0:0.5*\i)--++(-90:0.00) coordinate (c\i); 
			  }
		
		\foreach \i in {0,1,2,3,4,5,...,19}
		{
			\path (origin2)--++(0:0.25*\i) --++(-90:0.5) coordinate (b\i); 
			\path (origin2)--++(0:0.25*\i) --++(-90:0.9) coordinate (d\i); 
		}

	\path (4,-1) coordinate (origin); 
 		\path (origin)--++(0.5,0.5) coordinate (origin2);  
	 	\draw(origin2)--++(0:5.5) coordinate(X);
			 	\draw (origin2)--++(180:0.5)coordinate (Y); 	\draw[densely dotted](Y)--++(180:0.3); 
							 	\draw[densely dotted](X)--++(0:0.3); 
		\foreach \i in {0,1,2,3,4,5,...,10}
		{
			\path (origin2)--++(0:0.5*\i) coordinate (Da\i); 
			\path (origin2)--++(0:0.5*\i)--++(-90:0.00) coordinate (Dc\i); 
			  }
		
		\foreach \i in {1,2,3,4,5,...,19}
		{
			\path (origin2)--++(0:0.25*\i) --++(-90:0.5) coordinate (Db\i); 
			\path (origin2)--++(0:0.25*\i) --++(-90:0.9) coordinate (Dd\i); 
		}




		\draw[    thick](c8) to [out=-90, in=90] (Dc6);
		\draw[    thick](c10) to [out=-90, in=90] (Dc8);
		\draw[    thick](c12) to [out=-90, in=90] (Dc10);

		\draw[    thick](c-2) to [out=-90, in=90] (Dc0);

		\draw[    thick](c0) to [out=-90, in=90] (Dc2);
		\draw[    thick](c2) to [out=-90, in=90] (Dc4);
		\draw[    thick](c6) to [out=-90,in=0] (b10) to [out=180,in=-90] (c4); 
		\draw(c6) node [above] {\scalefont{0.7}$i{+}1$};
		\draw(c4) node [above] {\scalefont{0.7}$\phantom{{+}1}i\phantom{{+}1}$};		

%
%
%
%
%
%
%
					\path(Dc4) node [below] {\scalefont{0.7}$\phantom{{+}1}\phantom{{+}1}$};		
		\draw(1.75,0.5)	 node {$t_i^\ast=$}; 
	\end{tikzpicture}
$$where $t_i$ (respectively $t_i^*$) has $m+n-2$ vertices at the top (respectively bottom) and $m+n$ vertices at the bottom (respectively top). 
The bimodule $\mathbf{K}^{t_i}$ has basis given by all oriented diagrams of the form $\underline{\mu}\nu t_i \la \overline{\eta}$  for $\la, \eta\in \Lambda_{m-1,n-1}$ and $\mu , \nu \in \Lambda_{m,n}$. This is a left $K^m_n$-module and a right $K^{m-1}_{n-1}$-module where the actions are given by the surgery procedure described earlier.  The $(K^{m-1}_{n-1}, K^m_n)$-bimodule $\mathbf{K}^{t_i^*}$ is defined similarly. Tensoring with these bimodules gives rise to two functors
\begin{align*}
G^{t_i} \, : \, K^{m-1}_{n-1}\!\! - \!\! {\rm mod} & \rightarrow \Kmod  & G^{t_i^*} \, : \, \Kmod & \rightarrow K^{m-1}_{ n-1}\!\! - \!\! {\rm mod} \\ 
N & \mapsto \mathbf{K}^{t_i}\otimes_{K_{n-1}^{m-1}} N & 
M & \mapsto \mathbf{K}^{t_i^*}\otimes_{K_n^m} M
\end{align*}
for all $N\in  K^{m-1}_{n-1}\!\! - \!\! {\rm mod} $ and $M\in \Kmod$.
These functors satisfy the following properties.

\begin{thm}\label{exactadjoint}  \cite[Corollary 4.3, 4.4, 4.9, and Theorem 4.10]{MR2600694}
\begin{enumerate}
\item $\mathbf{K}^{t_i}$ is projective as a left $K^m_n$- and right $K^{m-1}_{n-1}$-module. Similarly, $\mathbf{K}^{t_i^*}$ is projective as a left $K_{n-1}^{m-1}$- and right $K_n^m$-module.
\item The functors $G^{t_i}$ and $G^{t_i^*}$ are exact and take projective modules to projective modules. 
\item The functors $G^{t_i}$ and $G^{t_i^*}$ commute with the duality functor $^{\ostar}$.
\item The functor $G^{t_i^*}$ is both left and right adjoint to the functor $G^{t_i}$, i.e. for any $X\in \Kmod$, $Y\in K^{m-1}_{n-1}\!\! - \!\! {\rm mod}$ we have
$$\Hom_{K^m_n}(X, G^{t_i}(Y)) \cong \Hom_{K^{m-1}_{n-1}}(G^{t_i^*}(X), Y)$$
and
$$\Hom_{K^m_n}(G^{t_i}(Y),X) \cong \Hom_{K^{m-1}_{n-1}}(Y,G^{t_i^*}(X)).$$
\end{enumerate}
\end{thm}

%

We also recall the effect of these functors on projective, simple and standard modules. Recall the notation from \cref{plusminusprime}.

\begin{thm}\label{Gsimplestandard} \cite[Theorem 4.2, 4.5, and 4.11]{MR2600694}
 Let $1\leq i <m+n$.
Suppose
$\la^+\in \Lambda_{m,n}^{\down \up}(i)$ with corresponding $\la'\in \Lambda_{m-1,n-1}$ and $\lambda^- \in \Lambda_{m,n}^{\up \down}(i)$. 
Then we have 
\begin{align}\label{Gsimplestandard1}
G^{t_i}P_{m-1,n-1}(\la') = P_{m,n}(\la^+).\end{align}
  There is a non-split exact sequence
\begin{align}\label{Gsimplestandard2}
0 \rightarrow \Delta_{m,n}(\la^-) \rightarrow G^{t_i}\Delta_{m-1,n-1}(\la') \rightarrow \Delta_{m,n}(\la^+)\rightarrow 0.
\end{align}
We also have
\begin{align}\label{Gsimplestandard3}
G^{t_i^*}\Delta_{m,n}(\mu) = \left\{ \begin{array}{ll} \Delta_{m-1,n-1}(\la') & \mbox{if $\mu = \la^+ \in \Lambda_{m,n}^{\down \up}(i)$ or $\mu = \la^-\in \Lambda_{m,n}^{ \up\down}(i)$} \\ 0 & \mbox{if $\mu \notin \Lambda_{m,n}^{\down \up}(i) \cup \Lambda_{m,n}^{\up \down}(i)$} \end{array}\right.
\end{align}
and 
\begin{align}\label{Gsimplestandard4}
G^{t_i^*}L_{m,n}(\mu) = \left\{ \begin{array}{ll} L_{m-1,n-1}(\la') & \mbox{if $\mu = \la^+\in \Lambda_{m,n}^{\down \up}(i)$}\\ 0 & \mbox{if $\mu \notin \Lambda_{m,n}^{\down \up}(i)$} \end{array} \right.
\end{align}
\end{thm}

\section{Tilting modules for $K^m_n$}

We are now ready to construct the tilting modules for the extended Khovanov arc algebras.

\begin{thm}\label{tilting}
For each $\la\in \Lambda_{m,n}$, define $T(\la) = T_{m,n}(\la)$ inductively as follows.
\begin{itemize}[leftmargin=*]
\item For $\la = (m^n)$ we define $T(m^n)=\Delta (m^n) = L(m^n)$. 
\item Now for $\la > (m^n)$ we have $\la = \la^- \in \Lambda^{\up \down}_{m,n}(i)$ for some $1\leq i < m+n$. Let $\la'\in \Lambda_{m-1,n-1}$ be the weight obtained from $\la$ by removing the symbols in position $i$ and $i+1$ and define 
$$T_{m,n}(\la) = G^{t_i}(T_{m-1, n-1}(\la')).$$ 
\end{itemize}
Then $T(\la)$ satisfies the following properties.
\begin{enumerate}[leftmargin=*]
\item $T(\la)^{\ostar} \cong T(\la)$.
\item We have an exact sequence
$$0 \rightarrow \Delta(\la) \rightarrow T(\la) \rightarrow J(\la) \rightarrow 0$$
with $J(\la) \in (\Kmod)^\Delta$ and $(J(\la):\Delta(\mu))\neq 0$ implies $\mu < \la$.
\item $T(\la)$ has simple socle isomorphic to $L(\la^\circ)$.
\end{enumerate}
Thus $T(\la)$ is the indecomposable tilting module with highest weight $\la$.
\end{thm}

\begin{proof} First note that as $(m^n)$ is minimal, we have $\Delta(m^n) = L(m^n)$.
 We have that (1) is immediate as $L(m^n)^{\ostar} \cong L(m^n)$ and the functor $G^{t_i}$ commutes with $\ostar$  by \cref{exactadjoint}(3).

We now consider (2). For $\la = (m^n)$, we have $T(m^n)=\Delta(m^n)=L(m^n)$ and $J(m^n)=\{0\}$. Now assume by induction that $T_{m-1, n-1}(\la')$ satisfies (2). So we have $(T_{m-1, n-1}(\la'):\Delta_{m-1,n-1}(\mu'))\neq 0$ implies $\mu' \leq \la'$ and $(T_{m-1, n-1}(\la'):\Delta_{m-1,n-1}(\la'))=1$.  Now applying the exact functor $G^{t_i}$ to each $\Delta (\mu')$ gives two factors $\Delta(\mu^-)$ and $\Delta(\mu^+)$ satisfying $\mu^+, \mu^- \leq \la = \la^-$.  Moreover, $\mu^- = \la^-$ if and only if $\mu' = \la'$, and so $(T_{m,n}(\la): \Delta_{m,n}(\la)) = 1$ as required.
Now, by induction we have $\Delta_{m-1,n-1}(\la')\hookrightarrow T_{m-1, n-1}(\la')$. Applying the exact functor $G^{t_i}$ we obtain
$$G^{t_i}(\Delta_{m-1,n-1}(\la'))\hookrightarrow T_{m, n}(\la).$$
Using \cref{Gsimplestandard2} gives the required inclusion.

Finally, we consider (3). We have $\soc T(m^n)=\soc L(m^n)=L(m^n)$ and by construction $(m^n)^\circ = (m^n)$ so the result holds for $\la = (m^n)$. Now for $\la > (m^n)$ we have
\begin{eqnarray*}
\Hom_{K^m_n}(L_{m,n}(\mu), T_{m,n}(\la)) &=& \Hom_{K^m_n}(L_{m,n}(\mu), G^{t_i}T_{m,n}(\la'))\\
&=& \Hom_{K^{m-1}_{n-1}}(G^{t_i^*}L_{m,n}(\mu), T_{m-1,n-1}(\la')).
\end{eqnarray*}
by \cref{exactadjoint}(4). 
Now using \cref{Gsimplestandard4} we have 
$$G^{t_i^*}L_{m,n}(\mu) = \left\{ \begin{array}{ll} L_{m-1, n-1}(\mu') & \mbox{if $\mu =\mu^+ \in \Lambda^{\down \up}_{m, n}(i)$}\\
0 & \mbox{otherwise} \end{array} \right. .$$
By induction we have 
$$
\Hom_{K^{m-1}_{n-1}}(L_{m-1,n-1}(\mu'), T_{m-1,n-1}(\la')) = \left\{ \begin{array}{ll} \Bbbk & \mbox{if $\mu' = (\la')^\circ$,}\\ 0 & \mbox{otherwise.} \end{array}\right .$$
Finally, note that if $\mu' = (\la')^\circ$ then $\mu = \mu^+ = \la^\circ$ by construction (remember that $\la = \la^- \in \Lambda^{\up \down}_{m,n}(i)$). This proves that $T_{m,n}(\la)$ has simple socle isomorphic to $L_{m,n}(\la^\circ)$ as required. 
\end{proof}

\section{Schur functor and the Khovanov arc algebra $H^m_n$}

We now introduce the original Khovanov arc algebras as idempotent truncations of the extended Khovanov arc algebras. 
Whilst this is anti-chronological, this way of presenting the material is more natural from our perspective and involves less doubling-up of  content.

\begin{thm}\cite[Theorem  6.1]{MR2600694}
\label{prinj} 
Let $\la \in \Lambda_{m,n}$. The following conditions are equivalent.
\begin{enumerate}
\item $\la \in \Lambda_{m,n}^\circ$.
\item $P(\la)^{\ostar} \cong P(\la)$.
\item $P(\la)$ is projective injective.
\end{enumerate}
\end{thm}

For each $\la\in \Lambda_{m,n}$, we set $e_\la = \underline{\la} \la \overline{\la} \in K^m_n$. Then it's easy to check that $e_\la^2 = e_\la$ and $e_\la e_\mu = 0$ when $\la \neq \mu$.  Note further that $e_\la^* = e_\la$.

\begin{defn}\cite[(6.8)]{MR2918294}
We define the {\sf Schur idempotent} to be the element
$$e = e_{m,n} = \sum_{\la \in \Lambda_{m,n}^\circ} e_\la\in K^m_n.$$
Then the {\sf Khovanov arc algebra} $H^m_n$ is defined to be the idempotent truncation
 $H^m_n := eK^m_n e.$
 \end{defn}
 
 \begin{rmk}\label{rmkmn2}
 Note that the algebra isomorphism $\curvearrowleft \circ ^* \, : \, K^m_n \rightarrow K_m^n$ given in \cref{rmkmn1} maps $e_{m,n}$ to $e_{n,m}$ and so we also have $H^m_n\cong H^n_m$.  
 \end{rmk}
 
We  have the corresponding Schur functor
\begin{align*}
f &= f_{m,n} \, : \, \Kmod \rightarrow \Hmod \, : \, M \mapsto eM
\intertext{and inverse Schur functors}
 g &= g_{m,n} \, : \, \Hmod \rightarrow \Kmod \, : \, N \mapsto \Hom_{H^m_n}(eK^m_n, N)
\intertext{and}
 \tilde{g} &= \tilde{g}_{m,n} \, : \, \Hmod \rightarrow \Kmod \, : \, N \mapsto K^m_n e\otimes_{H^m_n} N.
 \end{align*}
The functor $f$ is exact, the functor $g$ is left exact and the functor $\tilde{g}$ is right exact. Moreover, $f$ is left adjoint to $g$ and right adjoint to $\tilde{g}$. Also, as $e^*=e$ we have that all three functors commute with the duality functor $\ostar$.

\begin{thm}[{\cite[Proposition 6.3]{MR2600694}}]
The functor $f$ is fully faithful on $\Kproj$ and hence $(K^m_n, f)$ is a cover of $H^m_n$.
\end{thm}

Using \cref{prinj} we have that $H^m_n$ is a symmetric algebra with projective injective modules given by $fP(\la)$ for $\la \in \Lambda_{m,n}^\circ$, whose simple head and socle are given by $D(\la) = D_{m,n}(\la) := fL_{m,n}(\la)$. In fact, Brundan--Stroppel proved in \cite[Theorem 6.2]{MR2918294} that $H^m_n$ is a cellular algebra with cell modules given by
$S(\la) = S_{m,n}(\la) :=f\Delta_{m,n}(\la)$, $\la \in \Lambda_{m,n}$.  We denote by $(\Hmod)^S$ the subcategory of $\Hmod$ consisting of modules with a cell filtration.

\medskip

Using the bimodules $\mathbf{H}^{t_i} :=e_{m,n}\mathbf{K}^{t_i}e_{m-1,n-1}$ and $\mathbf{H}^{t_i^*} :=e_{m-1,n-1}\mathbf{K}^{t_i^*}e_{m,n}$, we obtain analogues of the projective functors for the Khonavov arc algebra, namely
\begin{align*}
&\overline{G}^{t_i} \, : \, H^{m-1}_{n-1}\!\! - \!\! {\rm mod} \rightarrow \Hmod \, : \, N \mapsto \mathbf{H}^{t_i}\otimes_{H_{n-1}^{m-1}} N 
\intertext{and}
& \overline{G}^{t_i^*} \, : \, \Hmod \rightarrow H^{m-1}_{n-1}\!\! - \!\! {\rm mod} \, : \, M \mapsto \mathbf{H}^{t_i^*}\otimes_{H_n^m} M
 \end{align*}
for all $N\in  H^{m-1}_{n-1}\!\! - \!\! {\rm mod} $ and $M\in \Hmod$.

\begin{prop}\label{barGexact} The bimodule $\mathbf{H}^{t_i}$ is projective as a left $H_n^m$- and right $H_{n-1}^{m-1}$-module. Similarly,  $\mathbf{H}^{t_i^*}$ is projective as a left $H_{n-1}^{m-1}$- and right $H_{n}^{m}$-module. Thus, the functors $\overline{G}^{t_i}$ and $ \overline{G}^{t_i^*}$ are exact and take projectives to projectives.  
\end{prop}

\begin{proof}
This follows directly from \cref{exactadjoint}(1) and  the definition of the bimodules $\mathbf{H}^{t_i}$ and $\mathbf{H}^{t_i^*}$.
\end{proof}

\begin{prop}\label{Gfiso} We have the following isomorphisms of functors.
\begin{enumerate}
\item  $\overline{G}^{t_i}f_{m-1,n-1} \cong f_{m,n} G^{t_i} \quad \mbox{and} \quad \overline{G}^{t_i^*}f_{m,n} \cong f_{m-1,n-1}G^{t_i^*}.$
\item  $G^{t_i}\tilde{g}_{m-1,n-1} \cong \tilde{g}_{m,n} \overline{G}^{t_i} \quad \mbox{and} \quad G^{t_i^*}\tilde{g}_{m,n} \cong \tilde{g}_{m-1,n-1}\overline{G}^{t_i^*}.$
\end{enumerate}
\end{prop}

\begin{proof}
This follows directly from the isomorphisms of functors
\begin{align*}
e_{m,n}\mathbf{K}^{t_i}\otimes_{K^{m-1}_{n-1}} (-) \cong \mathbf{H}^{t_i} \otimes_{H^{m-1}_{n-1}} e_{m-1,n-1}(-)
&
 \, : \, K_{n-1}^{m-1}\!\! - \!\! {\rm mod} \rightarrow \Hmod , 
\\
 e_{m-1,n-1}\mathbf{K}^{t_i^*}\otimes_{K^{m}_{n}} (-) \cong \mathbf{H}^{t_i^*} \otimes_{H^{m}_{n}} e_{m,n}(-) 
 &
 \, : \, \Kmod \rightarrow H^{m-1}_{n-1}\!\! - \!\! {\rm mod}.
 \end{align*}
given in \cite[(3.20)]{MR2600694}.
\end{proof}

\begin{cor}\label{baradjoint} For all $X\in K_{n-1}^{m-1}\!\! - \!\! {\rm mod}$ and $Y\in K_n^m\!\! - \!\! {\rm mod}$
we have 
\begin{align*}
\Hom_{H_{n}^{m}}(f_{m,n}(Y),\overline{G}^{t_i}f_{m-1, n-1}(X)) 
&
\cong \Hom_{H_{n-1}^{m-1}}(\overline{G}^{t_i^*}f_{m,n}(Y), f_{m-1,n-1}(X))
\intertext{and}
 \Hom_{H_n^m}(\overline{G}^{t_i}f_{m-1, n-1}(X), f_{m,n}(Y))
  &\cong \Hom_{H_{n-1}^{m-1}}(f_{m-1,n-1}(X), \overline{G}^{t_i^*}f_{m,n}(Y)).
 \end{align*}
\end{cor}

\begin{proof}
We start with the first statement. We have that 
\begin{align*}
\Hom_{H_{n}^{m}}(f_{m,n}(Y),\overline{G}^{t_i}f_{m-1, n-1}(X))
& \cong   \Hom_{H_{n}^{m}}(f_{m,n}(Y),f_{m, n}G^{t_i}(X))   
\\
&\cong   \Hom_{K^m_n}(\tilde{g}_{m,n}f_{m,n}(Y), G^{t_i}(X)) 
\\
&\cong   \Hom_{K^{m-1}_{n-1}}(G^{t_i^*}\tilde{g}_{m,n}f_{m,n}(Y), X) 
\\
&   
\cong 
  \Hom_{K^{m-1}_{n-1}}(\tilde{g}_{m-1,n-1}\overline{G}^{t_i^*}f_{m,n}(Y), X)  
  \\
&\cong 
  \Hom_{H^m_n}(\overline{G}^{t_i^*}f_{m,n}(Y), f_{m-1,n-1}(X)) 
\end{align*}
where the first isomorphism follows by \cref{Gfiso}(1); 
the second   because  $\tilde{g}_{m,n}$ is left adjoint to $f_{m,n}$;
the third   because  $G^{t_i^*}$ is adjoint to $G^{t_i}$;
the fourth   using \cref{Gfiso}(2); 
the fifth because  $f_{m-1,n-1}$ is right adjoint to $\tilde{g}_{m-1,n-1}$.

Now the second statement follows from the first using \cref{Gfiso} and the fact that the functors $G^{t_i}, G^{t_i^*}, f_{m,n}$ and $f_{m-1,n-1}$ commute with duality.
%
\end{proof}

Using \cref{barGexact} and \cref{baradjoint} we deduce the following version of Shapiro's lemma.

\begin{cor}\label{Extadjoint}
For all $X\in K_{n-1}^{m-1}\!\! - \!\! {\rm mod}$, $Y\in K_n^m\!\! - \!\! {\rm mod}$ and all $k\geq 0$
we have 
$${\rm Ext}^k_{H_{n}^{m}}(f_{m,n}(Y),\overline{G}^{t_i}f_{m-1, n-1}(X)) \cong {\rm Ext}^k_{H_{n-1}^{m-1}}(\overline{G}^{t_i^*}f_{m,n}(Y), f_{m-1,n-1}(X))$$
and
$${\rm Ext}^k_{H_n^m}(\overline{G}^{t_i}f_{m-1, n-1}(X), f_{m,n}(Y)) \cong {\rm Ext}^k_{H_{n-1}^{m-1}}(f_{m-1,n-1}(X), \overline{G}^{t_i^*}f_{m,n}(Y)).$$
\end{cor}

\section{$0$-faithfulness}

In this section we show that  $(K^m_n, f)$ is a $0$-faithful cover of $H^m_n$  if and only if   $n\neq m$. 
We will make use of  \cref{R0faithful}, which allows us to recast this question solely in terms of   tilting modules.

\begin{prop}\label{tiltingsequence} Assume $n\neq m$.
Then for all $\la \in \Lambda_{m,n}$, there is an exact sequence
$$0 \rightarrow T(\la) \rightarrow P^0 \rightarrow P^1$$
where $P^0$ and $P^1$ are projective-injective $K^m_n$-modules. 
\end{prop}

\begin{proof}
Using \cref{rmkmn1,rmkmn2}, we can assume  that $n>m$.
For $\la = (m^n)$ we have $T(m^n) = L(m^n)$. As $(m^n)\in \Lambda_{m,n}^\circ$, using \cref{prinj} we have $P(m^n)^{\ostar} \cong P(m^n)$  and so we get an exact sequence
$$0 \rightarrow L(m^n) \rightarrow P(m^n)$$
where $P(m^n)$ is projective-injective. We claim that 
$$
\soc(P(m^n)/L(m^n)) = \soc_2 P(m^n)
 = \bigoplus_{\mu \in \Gamma} L(\mu) \quad \mbox{for some $\Gamma \subseteq \Lambda_{m,n}^\circ$}.$$
Taking $P^0 = P(m^n)$ and $P^1 = \oplus_{\mu \in \Gamma}P(\mu)$ would then prove  the result for $\la = (m^n)$, once we verify the claim.

To prove the claim, we first observe that $P(m^n)  = T(m^{n-m})$. To see this, note that $P(m^n)$ is self-dual and has a standard filtration,  it is therefore an indecomposable tilting module. Moreover, using \cref{BH} and \cref{decnumbers}, we have that 
 the maximal $\mu$ such that 
 $$n_{\mu(m^n)}(1)= (P(m^n):\Delta(\mu))\neq0$$is given by $\mu=(m^{n-m})$,
and  therefore    $\Delta(m^{n-m}) \subseteq P(m^n)\cong  T(m^{n-m})$.  
We re-emphasise the fact that $\soc T(m^{n-m}) \cong L(m^n) \cong \soc \Delta(m^{n-m})$.  
Now using \cref{tilting} parts (2) and~(3) and \cref{soclestandard} 
 we have that 
$$
\soc_2 T(m^{n-m})\subseteq 
\soc_2 \Delta(m^{n-m}) \oplus \bigoplus _{\mu }
\soc \Delta(\mu)
\cong \soc_2 \Delta(m^{n-m}) \oplus \bigoplus _{\mu }L(\mu^\circ)$$
where both sums are over all $\mu\in \Lambda_{m,n}$ such that  $n_{ \mu (m^n) }(1)\neq0$. 
We have that $\mu^\circ \in \Lambda_{m,n}^\circ$ by \cref{lambdacirc}.  
Now, using \eqref{uniserial2}, we have that  $\Delta(m^{n-m})$ is uniserial of length $m+1$ with 
$$\rad_t \Delta(m^{n-m})=\soc_{m+1-t} \Delta(m^{n-m})
 \cong L(m^{n-m}, t^t)$$ for $0\leq t\leq m$. 
In particular, for $t=m-1$ this gives 
$$\soc_2 \Delta(m^{n-m}) \cong L(m^{n-m}, (m-1)^{m-1}).$$ Now, as $n>m$ we have $(m^{n-m}, (m-1)^{m-1})\in \Lambda_{m,n}^\circ$. This completes the proof of the claim and hence proves the result when $\la  =(m^n)$.

Now let $\la > (m^n)$ then we have $\la = \la^- \in \Lambda^{\up \down}_{m,n}(i)$ for some $1\leq i<m+n$. 
We can assume by induction that we have an exact sequence
$$0 \rightarrow T_{m-1,n-1}(\la') \rightarrow Q^0 \rightarrow Q^1$$
where $Q^0$ and $Q^1$ are projective-injective $K^{m-1}_{n-1}$-modules. Applying the exact functor $G^{t_i}$ we get an exact sequence
$$0 \rightarrow T_{m,n}(\la) \rightarrow G^{t_i}(Q^0) \rightarrow G^{t_i}(Q^1).$$
As $G^{t_i}$ takes projectives to projectives, and   commutes with the duality
 (see \cref{exactadjoint} parts (2) and (3)), both $G^{t_i}(Q^0)$ and $G^{t_i}(Q^1)$ are projective-injective $K^m_n$-modules as required. 
\end{proof}

\begin{thm}\label{0faithful} 
Assume $n\neq m$. Then  $(K^m_n, f)$ is a $0$-faithful cover of $H^m_n$, i.e.
$$\Hom_{K^m_n}(M,M') \cong \Hom_{H^m_n}(fM, fM')\quad
\mbox{for all $M, M'\in (\Kmod)^\Delta$.}$$
\end{thm}

\begin{proof}
We will use \cref{R0faithful} and prove that the counit $\eta$ is an isomorphism when evaluated on tilting modules. Clearly it is enough to prove this for all indecomposable tilting modules. 
Applying $gf$ to the exact sequence given in \cref{tiltingsequence} we get a commutative diagram with exact rows
$$
	   \begin{tikzpicture} [scale=0.8]
  \draw(-3.5,0) 	 node {$0$};	
    \draw(0,0) 	 node {$ T(\la)$};
  \draw(4,0) 	 node {$  P^0$};    
  \draw(8,0) 	 node {$  P^1$};

  \draw(-3.5,-3) 	 node {$0$};	
    \draw(0,-3) 	 node {$ gfT(\la)$};
  \draw(4,-3) 	 node {$ gf P^0$};    
  \draw(8,-3) 	 node {$gf  P^1$};

\draw[->](1,-3)--(3,-3);
\draw[->](-3,-3)--(-1,-3);
\draw[->](1+4,-3)--(3+4,-3);
\draw[->](1,0)--(3,0);
\draw[->](-3,0)--(-1,0);
\draw[->](1+4,0)--(3+4,0) ;
\draw(6,0.3) node {\scalefont{0.9}$\vartheta$};

\draw[->](0,-0.5)--(0,-2.5);\draw[->](0+4,-0.5)--(0+4,-2.5);\draw[->](0+4+4,-0.5)--(0+4+4,-2.5);

\draw(0,-1.5) node[right] {\scalefont{0.9}$\eta(T(\la))$};
\draw(4,-1.5) node[right] {\scalefont{0.9}$\eta(P^0)$};
\draw(8,-1.5) node[right] {\scalefont{0.9}$\eta(P^1)$};
	\end{tikzpicture}	$$ 
	As the map $\eta(P^0)$ is an isomorphism we have that the map $\eta(T(\la))$ must be injective. Moreover, as $\eta(P^0)$ and $\eta(P^1)$ are both isomorphisms, we have 
 $$\dim (T(\la)) = \dim ({\rm Ker}(\vartheta)) = \dim( {\rm Ker}( gf\vartheta) )= \dim (gfT(\la))$$
 and so $\eta(T(\la))$ must be an isomorphism.
\end{proof}

 \begin{cor}
The extended arc algebras   $K^m_n $ are $0$-faithful covers of the 
Khovanov arc algebras $H^m_n $ if and only if $m\neq n$.  
\end{cor}
 
 \begin{proof}One direction is immediate from \cref{0faithful}.  To see that 0-faithfulness fails when $m=n$, we observe that $ S_{m,m}(m^m) \cong D_{m,m}(m^m)\cong S_{m,m}$ whereas 
 \begin{align*}
 \Hom_{K^m_m }(\Delta_{m,m}(\varnothing), \Delta_{m,m}(m^m))=0
  \end{align*} 
since $\rad_0(\Delta_{m,m}(\varnothing))=L_{m,m}(\varnothing)$ is not a composition factor of 
$\Delta_{m,m}(m^m))$ by \cref{decnumbers}.
 \end{proof}

%

\section{$(|n-m|-1)$-faithfulness}

We are now ready to prove the main result of the paper: that the extended arc algebras   $K^m_n $ are $(|n-m|-1)$-faithful covers of the
Khovanov arc algebras $H^m_n $. 
Throughout this section we assume without loss of generality that $n>m$.

\begin{lem}\label{lem1}
For all $0\leq j < n-m$ we have
\begin{align}\label{(1)}{\rm Ext}^j_{H^m_n}(D(m^m), D(m^n)) = 0.
\intertext{For all $0<j<n-m$ and $\la \in \Lambda_{m,n}$ we have }
 \label{(2)} {\rm Ext}^j_{H^m_n}(D(m^m), fT(\la)) = 0.
 \end{align}
\end{lem}

\begin{proof}
We first verify \eqref{(1)}. As the simple modules are self-dual we have 
$${\rm Ext}^j_{H^m_n}(D(m^m), D(m^n)) \cong {\rm Ext}^j_{H^m_n}(D(m^n), D(m^m)).$$ 
We will show that the latter is zero for $0\leq j <n-m$. Consider the projective resolution of $\Delta(m^n) = L(m^n)$ given in \cref{projresol} 
$$\cdots \rightarrow P_2 \rightarrow  P_1 \rightarrow P_0 = P(m^n) \rightarrow L(m^n) \rightarrow 0.$$
We claim that $P_k$ is a direct sum of $P(\mu)$'s for $\mu \in \Lambda_{m,n}^\circ$ for all $0\leq k \leq n-m$. Using \cref{inverseKL2}, we have that if $P(\mu)$ is a direct summand of $P_k$ then we have 
$$(m^n)=\la_0 \rightarrow \la_1 \rightarrow \ldots \rightarrow \la_k = \mu.$$
Now using   \cref{littleclaim}  we have 
\begin{eqnarray*}
\min \{ \ell_h(\la_k) \, : \, \mbox{$h$  is labelled by $\up$ in $\la_k$}\} &\geq & \min \{ \ell_h(m^n) \, : \, \mbox{$h$  is labelled by $\up$ in $(m^n)$}\}  -k \\
&=& n-m-k \geq 0
\end{eqnarray*}
for $0\leq k \leq n-m$. Finally note that $\ell_h(\la_k)\geq 0$ for all $1\leq h\leq m+n$ if and only if this is the case for all $h$ labelled by $\up$ in $\la_k$. This proves the claim that $\mu = \la_k\in \Lambda_{m,n}^\circ$. 

Now applying the Schur functor to the projective resolution above we get an exact sequence
$$\cdots \rightarrow fP_2 \rightarrow  fP_1 \rightarrow fP_0 = fP(m^n) \rightarrow D(m^n) \rightarrow 0.$$
where $fP_k$ is a projective $H^m_n$-module for each $0\leq k\leq n-m$.
By elementary homological algebra we have 
$$ {\rm Ext}^t_{H^m_n}(D(m^n), D(m^m)) = 0 \quad 
 \Leftrightarrow \quad {\rm Ext}^j_{H^m_n}(fP_k, D(m^m))= 0 \quad \mbox{for all $j+k\leq t$}.$$
If $j\geq 1$ and $k\leq n-m$ we have ${\rm Ext}^j_{H^m_n}(fP_k, D(m^m))=0$ as $fP_k$ is projective. 
If $j=0$ then $\Hom_{H^m_n}(fP_k, D(m^m))=0$ precisely when $D(m^m)$ does not appear in the head of $fP_k$.  Now we have 
$$\rad_0 fP_k= \bigoplus_{\mu}p_{(m^n)\mu}^{(k)}D(\mu)$$
and we have $p_{(m^n)(m^m)}^{(k)}\neq 0$ if and only if $k=m(n-m)$ using \cref{inverseKL1}. Thus we have 
$\Hom_{H^m_n}(fP_k, D(m^m))=0$ for all $0\leq k < n-m$ as required.

We now verify \eqref{(2)}. If $\la =(m^n)$ is minimal then $T(m^n)=L(m^n)$ and so $fT(m^n)=D(m^n)$ and the result holds by \eqref{(1)}. Now assume that $\la$ is not minimal, then $\la \in \Lambda_{m,n}^{\up \down}(i)$ for some $i\in \ZZ_{>0}$. Let $\la'\in \Lambda_{m-1, n-1}$ be the weight obtained from $\la$ by removing the symbols in position $i$ and $i+1$. Then we have $T(\la) = G^{t_i}(T(\la'))$. Now, using \cref{Gfiso}(1), \cref{Extadjoint} and \cref{Gsimplestandard4} we have 
\begin{eqnarray*} {\rm Ext}^j_{H^m_n}(D(m^m), fT(\la)) &\cong&  {\rm Ext}^j_{H^m_n}(D(m^m), \overline{G}^{t_i}fT(\la')) \\
&\cong & {\rm Ext}^j_{H^{m-1}_{n-1}}(\overline{G}^{t_i^*}fL(m^m), fT(\la')) \\
&\cong & {\rm Ext}^j_{H^{m-1}_{n-1}}(fG^{t_i^*}L(m^m), fT(\la')) \\
&\cong & \left\{ \begin{array}{ll}  {\rm Ext}^j_{H^{m-1}_{n-1}}(D((m-1)^{m-1}), fT(\la'))  & \mbox{if $i=m$}\\ 0 &  \mbox{if $i\neq m$} \end{array}\right. \\
&=& 0 
\end{eqnarray*}
for $0<j<n-m$ by induction.
\end{proof}

\begin{lem}\label{lem2} Assume $n>m$ and $\la \neq \varnothing$  we have that 
$$\Hom_{H^m_n}(D(m^m), S(\la)) = 0.$$
\end{lem}

\begin{proof}
By \cref{soclestandard} we have that $\soc \Delta(\la)=L(\la^\circ)$ and by \cref{prinj} 
we have that $P(\la^\circ)$ is projective-injective; therefore 
we have an embedding $\Delta(\la) \hookrightarrow P(\la^\circ)$. Applying the Schur functor $f$ we get 
$$S(\la) \hookrightarrow fP(\la^\circ).$$
Now, as $\la^\circ \in \Lambda_{m,n}^\circ$, we have that $fP(\la^\circ)$ is indecomposable projective-injective and hence has simple socle $D(\la^\circ)$. It remains to show that if $\la^\circ = (m^m)$ then we must have $\la = \varnothing$. Now, the cup diagram $\underline{(m^m)}$ is given by $m$ concentric cups, followed by $n-m$ rays (see \cref{Figlem2}). We claim that when $\la^\circ = \underline{(m^m)}$, all cups in $\underline{\la^\circ}\la$ must be clockwise. This would imply that $\la = \varnothing$. 
Suppose, for a contradiction that $\underline{\la^\circ}\la$ has at least one anti-clockwise cup. Then, by construction, we must have that the outermost cup, connecting vertices $1$ and $2m$, must be anti-clockwise. This means that the vertices $2m$ and $2m+1$ of $\la$ are labelled with $\up$ and $\down$, respectively (see \cref{Figlem2}). But this contradicts the fact that $\la^\circ = (m^m)$ as the vertices $2m$ and $2m+1$ should be connected with a cup in $\la^\circ$. \end{proof}

 
 \begin{figure}[ht!]
 $$   \begin{tikzpicture} [scale=0.75]

		\path (4,1) coordinate (origin); 
 		\path (origin)--++(0.5,0.5) coordinate (origin2);  
	 	\draw(origin2)--++(0:1.25) coordinate (origin3); 
	 	\draw[densely dotted](origin3)--++(0:.5) coordinate (origin3); 	
	 	\draw(origin3)--++(0:2) coordinate (origin3); 
 	 	\draw[densely dotted](origin3)--++(0:.5) coordinate (origin3); 			
	\draw(origin3)--++(0:2) coordinate (origin3); 		
			 	\draw[densely dotted](origin3)--++(0:.5) coordinate (origin3); 				
	\draw(origin3)--++(0:1) coordinate (origin3); 		 						
		\foreach \i in {1,2,3,4,5,...,15}
		{
			\path (origin2)--++(0:0.5*\i) coordinate (a\i); 
			\path (origin2)--++(0:0.5*\i)--++(-90:0.00) coordinate (c\i); 
			  }
		
		\foreach \i in {1,2,3,4,5,...,19}
		{
			\path (origin2)--++(0:0.25*\i) --++(-90:0.5) coordinate (b\i); 
			\path (origin2)--++(0:0.25*\i) --++(-90:0.9) coordinate (d\i); 
		}
\foreach \i in {1,2,4,5,11,12,14,15}
{\path(a\i)--++(90:0.12) node  {  $  \down   $} ;
}

\foreach \i in {6,7,9,10}
{\path(a\i)--++(-90:0.12) node  {  $  \up  $} ;
}


		\draw[    thick](c6) to [out=-90,in=0] (b11) to [out=180,in=-90] (c5); 		
	
			\path(b11)--++(-90:0.2) coordinate (b11);

		\draw[    thick](c7) to [out=-90,in=0] (b11) to [out=180,in=-90] (c4);


		\path(b11)--++(-90:0.8) coordinate (b11);

		\draw[    thick](c10) to [out=-90,in=0] (b11) to [out=180,in=-90] (c1);

		\path(b11)--++(90:0.3) coordinate (b11);

		\draw[    thick](c9) to [out=-90,in=0] (b11) to [out=180,in=-90] (c2);

		\draw[ thick](c12)--++(-90:0.85) coordinate(X);
			\draw[ thick,densely   dotted](X)--++(-90:0.3);
		\draw[ thick](c11)--++(-90:0.85) coordinate(X);
			\draw[ thick,densely   dotted](X)--++(-90:0.3);	
				\draw[ thick](c14)--++(-90:0.85) coordinate(X);
							\draw[ thick,densely   dotted](X)--++(-90:0.3);
		\draw[ thick](c15)--++(-90:0.85) coordinate(X);
			\draw[ thick,densely   dotted](X)--++(-90:0.3);

\draw(3.8,1.5) node {$(m^m)$};

\draw(3.8,0.75) node {$ \underline{(m^m)}$};

	\path (13.5,1) coordinate (origin); 
 		\path (origin)--++(0.5,0.5) coordinate (origin2);  
	 	\path(origin2)--++(0:1.25) coordinate (origin3); 
	 	\path[densely dotted](origin3)--++(0:.5) coordinate (origin3); 	
	 	\path(origin3)--++(0:2) coordinate (origin3); 
 	 	\path[densely dotted](origin3)--++(0:.5) coordinate (origin3); 			
	\path(origin3)--++(0:2) coordinate (origin3); 		
			 	\path[densely dotted](origin3)--++(0:.5) coordinate (origin3); 				
	\path(origin3)--++(0:1) coordinate (origin3);

		\end{tikzpicture}
$$

 $$   \begin{tikzpicture} [scale=0.75]

		\path (4,1) coordinate (origin); 
 		\path (origin)--++(0.5,0.5) coordinate (origin2);  
	 	\draw(origin2)--++(0:1.25) coordinate (origin3); 
	 	\draw[densely dotted](origin3)--++(0:.5) coordinate (origin3); 	
	 	\draw(origin3)--++(0:2) coordinate (origin3); 
 	 	\draw[densely dotted](origin3)--++(0:.5) coordinate (origin3); 			
	\draw(origin3)--++(0:2) coordinate (origin3); 		
			 	\draw[densely dotted](origin3)--++(0:.5) coordinate (origin3); 				
	\draw(origin3)--++(0:1) coordinate (origin3); 		 						
		\foreach \i in {1,2,3,4,5,...,15}
		{
			\path (origin2)--++(0:0.5*\i) coordinate (a\i); 
			\path (origin2)--++(0:0.5*\i)--++(-90:0.00) coordinate (c\i); 
			  }
		
		\foreach \i in {1,2,3,4,5,...,19}
		{
			\path (origin2)--++(0:0.25*\i) --++(-90:0.5) coordinate (b\i); 
			\path (origin2)--++(0:0.25*\i) --++(-90:0.9) coordinate (d\i); 
		}
\foreach \i in {1,11,12,14,15}
{\path(a\i)--++(90:0.12) node  {  $  \down   $} ;
}

\foreach \i in {10}
{\path(a\i)--++(-90:0.12) node  {  $  \up  $} ;
}


		\draw[    thick](c6) to [out=-90,in=0] (b11) to [out=180,in=-90] (c5); 		
	
			\path(b11)--++(-90:0.2) coordinate (b11);

		\draw[    thick](c7) to [out=-90,in=0] (b11) to [out=180,in=-90] (c4);


		\path(b11)--++(-90:0.8) coordinate (b11);

		\draw[  red,  thick](c10) to [out=-90,in=0] (b11) to [out=180,in=-90] (c1);

		\path(b11)--++(90:0.3) coordinate (b11);

		\draw[    thick](c9) to [out=-90,in=0] (b11) to [out=180,in=-90] (c2);

		\draw[ thick](c12)--++(-90:0.85) coordinate(X);
			\draw[ thick,densely   dotted](X)--++(-90:0.3);
		\draw[ thick](c11)--++(-90:0.85) coordinate(X);
			\draw[ thick,densely   dotted](X)--++(-90:0.3);	
				\draw[ thick](c14)--++(-90:0.85) coordinate(X);
							\draw[ thick,densely   dotted](X)--++(-90:0.3);
		\draw[ thick](c15)--++(-90:0.85) coordinate(X);
			\draw[ thick,densely   dotted](X)--++(-90:0.3);

\draw(3.8,1.5) node {$\la$};

\draw(3.8,0.75) node {$ \underline{(m^m)}$};		 
%
%
		\path (13.5,1) coordinate (origin); 
 		\path (origin)--++(0.5,0.5) coordinate (origin2);  
	 	\draw(origin2)--++(0:1.25) coordinate (origin3); 
	 	\draw[densely dotted](origin3)--++(0:.5) coordinate (origin3); 	
	 	\draw(origin3)--++(0:2) coordinate (origin3); 
 	 	\draw[densely dotted](origin3)--++(0:.5) coordinate (origin3); 			
	\draw(origin3)--++(0:2) coordinate (origin3); 		
			 	\draw[densely dotted](origin3)--++(0:.5) coordinate (origin3); 				
	\draw(origin3)--++(0:1) coordinate (origin3); 		 						
		\foreach \i in {1,2,3,4,5,...,15}
		{
			\path (origin2)--++(0:0.5*\i) coordinate (a\i); 
			\path (origin2)--++(0:0.5*\i)--++(-90:0.00) coordinate (c\i); 
			  }
		
		\foreach \i in {1,2,3,4,5,...,25}
		{
			\path (origin2)--++(0:0.25*\i) --++(-90:0.5) coordinate (b\i); 
			\path (origin2)--++(0:0.25*\i) --++(-90:0.9) coordinate (d\i); 
		}
\foreach \i in {1,11,12,14,15}
{\path(a\i)--++(90:0.12) node  {  $  \down   $} ;
}

\foreach \i in {10}
{\path(a\i)--++(-90:0.12) node  {  $  \up  $} ;
}
 		
		\draw[    thick](c11) to [out=-90,in=0] (b21) to [out=180,in=-90] (c10);

%
%
%
%
%
%
%
%
%

		\draw[ thick](c12)--++(-90:0.85) coordinate(X);
			\draw[ thick,densely   dotted](X)--++(-90:0.3);
				\draw[ thick](c14)--++(-90:0.85) coordinate(X);
							\draw[ thick,densely   dotted](X)--++(-90:0.3);
		\draw[ thick](c15)--++(-90:0.85) coordinate(X);
			\draw[ thick,densely   dotted](X)--++(-90:0.3);

\path(origin2)--++(180:0.2) node {$\la$};

\path(origin2)--++(180:0.2)--++(-90:0.75)node {$ \underline{\la^\circ}$};		 
	\end{tikzpicture}$$

	 $$   \begin{tikzpicture} [scale=0.75]

		\path (4,1) coordinate (origin); 
 		\path (origin)--++(0.5,0.5) coordinate (origin2);  
	 	\draw(origin2)--++(0:1.25) coordinate (origin3); 
	 	\draw[densely dotted](origin3)--++(0:.5) coordinate (origin3); 	
	 	\draw(origin3)--++(0:2) coordinate (origin3); 
 	 	\draw[densely dotted](origin3)--++(0:.5) coordinate (origin3); 			
	\draw(origin3)--++(0:2) coordinate (origin3); 		
			 	\draw[densely dotted](origin3)--++(0:.5) coordinate (origin3); 				
	\draw(origin3)--++(0:1) coordinate (origin3); 		 						
		\foreach \i in {1,2,3,4,5,...,15}
		{
			\path (origin2)--++(0:0.5*\i) coordinate (a\i); 
			\path (origin2)--++(0:0.5*\i)--++(-90:0.00) coordinate (c\i); 
			  }
		
		\foreach \i in {1,2,3,4,5,...,19}
		{
			\path (origin2)--++(0:0.25*\i) --++(-90:0.5) coordinate (b\i); 
			\path (origin2)--++(0:0.25*\i) --++(-90:0.9) coordinate (d\i); 
		}
\foreach \i in {6,7,9,10,11,12,14,15}
{\path(a\i)--++(90:0.12) node  {  $  \down   $} ;
}

\foreach \i in {1,2,4,5}
{\path(a\i)--++(-90:0.12) node  {  $  \up  $} ;
}


		\draw[    thick](c6) to [out=-90,in=0] (b11) to [out=180,in=-90] (c5); 		
	
			\path(b11)--++(-90:0.2) coordinate (b11);

		\draw[    thick](c7) to [out=-90,in=0] (b11) to [out=180,in=-90] (c4);


		\path(b11)--++(-90:0.8) coordinate (b11);

		\draw[    thick](c10) to [out=-90,in=0] (b11) to [out=180,in=-90] (c1);

		\path(b11)--++(90:0.3) coordinate (b11);

		\draw[    thick](c9) to [out=-90,in=0] (b11) to [out=180,in=-90] (c2);

		\draw[ thick](c12)--++(-90:0.85) coordinate(X);
			\draw[ thick,densely   dotted](X)--++(-90:0.3);
		\draw[ thick](c11)--++(-90:0.85) coordinate(X);
			\draw[ thick,densely   dotted](X)--++(-90:0.3);	
				\draw[ thick](c14)--++(-90:0.85) coordinate(X);
							\draw[ thick,densely   dotted](X)--++(-90:0.3);
		\draw[ thick](c15)--++(-90:0.85) coordinate(X);
			\draw[ thick,densely   dotted](X)--++(-90:0.3);

\draw(3.8,1.5) node {$\varnothing$};

\draw(3.8,0.75) node {$  \underline{(m^m)}$};

	\path (13.5,1) coordinate (origin); 
 		\path (origin)--++(0.5,0.5) coordinate (origin2);  
	 	\path(origin2)--++(0:1.25) coordinate (origin3); 
	 	\path[densely dotted](origin3)--++(0:.5) coordinate (origin3); 	
	 	\path(origin3)--++(0:2) coordinate (origin3); 
 	 	\path[densely dotted](origin3)--++(0:.5) coordinate (origin3); 			
	\path(origin3)--++(0:2) coordinate (origin3); 		
			 	\path[densely dotted](origin3)--++(0:.5) coordinate (origin3); 				
	\path(origin3)--++(0:1) coordinate (origin3);

		\end{tikzpicture}
$$

 \caption{An illustration of the proof that $\underline{\la^\circ} = \underline{(m^m)}$ implies that $\la = \varnothing$ when $n>m$}\label{Figlem2}

  \end{figure}

\begin{prop}\label{derived} Assume $n>m$.
For all $Y\in (\Hmod)^S$ we have
$$R^jgY = 0 \quad \mbox{for all $0<j<n-m$}.$$
\end{prop}

\begin{proof}
For any $X\in \Kmod$, $Y\in \Hmod$ we have a Grothendieck spectral sequence with second page 
$${\rm Ext}^i_{K^m_n}(X, R^jgY)\Rightarrow {\rm Ext}^{i+j}_{H^m_n}(fX, Y).$$
If $X=P\in \Kproj$ then this degenerates to give
$$\Hom_{K^m_n}(P, R^jgY)\cong {\rm Ext}^j_{H^m_n}(fP, Y).$$
Now we have $R^jgY=0$ if and only if 
$$0 = \Hom_{K^m_n}(P(\la), R^jgY)\cong {\rm Ext}^j_{H^m_n}(fP(\la), Y) \quad \mbox{for all $\la \in \Lambda_{m,n}$.}$$
So we need to show that for any $Y\in (\Hmod)^S$ and any $\la \in \Lambda_{m,n}$ we have
$${\rm Ext}^j_{H^m_n}(fP(\la), Y) = 0 \quad \mbox{for all $0<j<n-m$}.$$
Clearly, it is enough to show that for all $\la, \mu \in \Lambda_{m,n}$ we have
$${\rm Ext}^j_{H^m_n}(fP(\la), S(\mu)) = 0 \quad \mbox{for all $0<j<n-m$}.$$
We prove this by downward  induction on $\la \in \Lambda_{m,n}$. If $\la = \varnothing$ is maximal then $P(\varnothing) = \Delta(\varnothing)$ and, using \eqref{uniserial1} we have
$f\Delta(\varnothing) \cong  D(m^m)$. So we need to prove that
$${\rm Ext}^j_{H^m_n}(D(m^m), S(\mu)) = 0 \quad \mbox{for all $0<j<n-m$}.$$
We use induction on $\mu$. If $\mu = (m^n)$ is minimal then $S(m^n) = D(m^n)$ and we are done by \eqref{(1)}. 
Now let $\mu > (m^n)$ and assume that the result holds for all $\nu < \mu$. Using \cref{tilting} we  have a short exact sequence
$$0 \rightarrow \Delta(\mu) \rightarrow T(\mu) \rightarrow J(\mu) \rightarrow 0$$
where $(J(\mu):\Delta(\nu))\neq 0$ implies $\nu <\mu$. Applying the functor $\Hom_{H^m_n}(D(m^m), f(-))$ we obtain a long exact sequence
$$\ldots \rightarrow {\rm Ext}^{j-1}_{H^m_n}(D(m^m), fJ(\mu)) \rightarrow {\rm Ext}^j_{H^m_n}(D(m^m), S(\mu)) \rightarrow {\rm Ext}^j_{H^m_n}(D(m^m), fT(\mu)) \rightarrow \ldots$$
For $j>1$ we have that ${\rm Ext}^{j-1}_{H^m_n}(D(m^m), fJ(\mu)) =0$ by induction. For $j=1$ we have 
$${\rm Ext}^{j-1}_{H^m_n}(D(m^m), fJ(\mu)) = \Hom_{H^m_n}(D(m^m), fJ(\mu)) = 0$$
by \cref{lem2} as $(J(\mu):\Delta(\nu))\neq 0$ implies $\nu < \mu$ and so $\nu \neq \varnothing$. 
Finally, using \eqref{(2)} we have 
$$ {\rm Ext}^j_{H^m_n}(D(m^m), fT(\mu))  = 0.$$
Thus we must have 
${\rm Ext}^j_{H^m_n}(D(m^m), S(\mu)) = 0$ for all $0<i<n-m$ as required. This completes the case $\la = \varnothing$.

Now let $\varnothing > \la\in \Lambda_{m,n}$ and assume the result holds for any weight in $\Lambda_{m-1,n-1}$.  Then there exists $i$ such that $\la \in \Lambda_{m,n}^{\down \up}(i)$. Let $\la'\in \Lambda_{m-1,n-1}$  be the weight obtained from $\la$ by removing the symbols in positions $i$ and $i+1$. Then using \cref{Gsimplestandard1} we have 
$P_{m,n}(\la) = G^{t_i}(P_{m-1, n-1}(\la'))$. Using \cref{Gfiso} and \cref{Extadjoint} we have, for each $0<j<n-m$, that 
\begin{eqnarray*}
{\rm Ext}^j_{H^m_n}(fP_{m,n}(\la) ,S(\mu) ) &=& {\rm Ext}^j_{H^m_n}(\overline{G}^{t_i}fP_{m-1,n-1}(\la'), f\Delta(\mu))\\
&\cong & {\rm Ext}^j_{H^{m-1}_{n-1}}(fP_{m-1,n-1}(\la'), fG^{t_i^*}(\Delta(\mu)))\\
&=& 0.
\end{eqnarray*}
The last equality follows by induction as $fG^{t_i^*}\Delta(\mu)\in (H^{m-1}_{n-1}\!\!-\!\!{\rm mod})^S$ by \cref{Gsimplestandard3}.
\end{proof}

\begin{thm}\label{ifaithful}
Assume $n\neq m$. Then for any $X\in \Kmod$ and $M\in (\Kmod)^\Delta$ we have
$${\rm Ext}^j_{K^m_n}(X, M)\cong {\rm Ext}^j_{H^m_n}(fX, fM) \quad \mbox{for all $0\leq j<|n-m|$}.$$
In particular, $(K^m_n, f)$ is an $(|n-m|-1)$-faithful cover of $H^m_n$.
\end{thm}

\begin{proof} Assume without loss of generality that $n>m$.
For $X\in \Kmod$ and $Y\in \Hmod$, we have a Grothendieck spectral sequence with second page 
$${\rm Ext}^j_{K^m_n}(X,R^igY) \Rightarrow {\rm Ext}^{i+j}_{H^m_n}(fX,Y).$$
Now when $Y\in (\Hmod)^S$, using \cref{derived}, we obtain
$${\rm Ex}t^j_{K^m_n}(X, gY) \cong {\rm Ext}^j_{H^m_n}(fX, Y) \quad \mbox{for all $0\leq j<n-m$}.$$
Now take $Y = fM$ for some $M\in (\Kmod)^\Delta$.  Using \cref{0faithful}, we have $gY=gfM\cong M$ which gives
$${\rm Ex}t^j_{K^m_n}(X, M) \cong {\rm Ext}^j_{H^m_n}(fX, fM) \quad \mbox{for all $0\leq j<n-m$}$$
as required.
\end{proof}

\begin{cor}
For $|n-m|\geq 2$ the functor $f$ induces an equivalence of exact categories from $(\Kmod)^\Delta$ to $(\Hmod)^S$ with inverse $g$.
\end{cor}

\begin{proof}
This follows from \cref{ifaithful} and \cite[Proposition 4.41]{ROUQ}.
\end{proof}

\begin{rmk}
All the results of this paper 
concern the passage of cohomological information between the (sub)categories $(\Kmod)^\Delta$ and $(\Hmod)^S$. 
In our companion paper \cite{BDDHMS2} we will consider the passage of information between
$\Kmod $ and $\Hmod$. 
\end{rmk}

\begin{rmk}
In   \cite[Corollary 8.6]{MR2781018} and \cite[B6. Corollary]{hm15} they prove that the
(basic algebra of the) 
 unique non-semisimple block of the cyclotomic  quiver Hecke (respectively  
Schur algebra)  of level 2, rank $mn$ and  quantum characteristic $e>m+n$
 is isomorphic to the    (extended) Khovanov arc algebra  $H^m_n$ ($K^m_n$ respectively).  
Combinatorially this matches up $\la\in \Lambda_{m,n}$ of this paper with a pair of partitions $\varphi_{m,n}(\la)=(\la,\la^c)$ such that  $\la+\la^c=(m^n)$. 
Therefore our faithfulness results immediately transfer to this setting. 
In this remark, we wish to emphasise that  the ``Scopes-esque" Morita equivalences constructed in 
 \cite{Scopes, WebAffRoCK} allow us to widen this faithfulness result  to a much broader class of blocks for quiver Hecke/Schur algebras.  In \cite{WebAffRoCK} the notion of a ``RoCK" block is defined for the  quiver Hecke/Schur algebras (and more general categorifications) and their study was initiated in  \cite{LyleARRoCK, LyleAR1, WebAffRoCK, MNSS, DA} motivated by the success of this theory in level 1 \cite{MR1874244, MR3862945}. 
 Amongst these RoCK blocks the core blocks of Fayers \cite{MR2335702} are the simplest: for level~1 they are semisimple and for level 2 we claim that they are all Morita equivalent to the (extended) Khovanov arc algebras.  
 
 To make this statement more precise, one may associate to a level two core block \(R^\Omega_\beta(\mathfrak{sl}_{e}, \Bbbk)\) a pair of integers \(0 \leq m_{\Omega, \beta} \leq n_{\Omega, \beta}< e\). If \(t_+\) and \(t_-\) are the number of \(+\)'s and \(-\)'s in the {\em sign sequence} for any bipartition in the block \(R^\Omega_\beta\) (see \cite{LyleARB}), then \(m_{\Omega,\beta} = \min\{t_+, t_-\}\) and \(n_{\Omega, \beta} = \max\{t_+, t_-\}\). The integer \(m_{\Omega, \beta}\) is the   {\sf weight}  of the core block (see \cite{MR2335702}), and we have \(m_{\Omega, \beta} + n_{\Omega, \beta}\leq e \) and \( n_{\Omega, \beta} - m_{\Omega, \beta} \in \{|i-j|, e-|i-j|\}\) when \(\Omega = \Omega_i + \Omega_j\). 
Now it is straightforward to show, thanks to combinatorial descriptions of core blocks in level two \cite{MR2335702, LyleARB} and applications of {\em Scopes} equivalences \cite{Scopes, WebAffRoCK} which preserve \(t_+, t_-\), that the core block \(R^{\Omega}_\beta(\mathfrak{sl}_{e})\) is Morita equivalent to a   block \(R^{\Omega}_{\gamma'}(\mathfrak{sl}_{e})\), which contains
the bipartition  \(\varphi_{m,n}(\varnothing)=(\varnothing, (m_{\Omega, \beta}^{n_{\Omega,\beta}}))\). 
Thus all core blocks of quiver Hecke (respectively Schur) algebras are Scopes-equivalent to the  
(extended) Khovanov arc algebra  $H^m_n$ ($K^m_n$ respectively).
\end{rmk}

               \bibliographystyle{amsalpha}   
\bibliography{master}

 \end{document}